\documentclass[final, a4paper, 11pt]{article}

\usepackage[scale=0.768]{geometry}
\usepackage[english]{babel}

\usepackage[mathscr]{euscript}
\usepackage{color}
\usepackage{fancyhdr}
\usepackage{dsfont}
\usepackage{bm}
\usepackage{hyperref}
\usepackage{cite}
\usepackage{showkeys}
\usepackage{subfigure}
\usepackage{float}

\usepackage{amsfonts}
\usepackage{amsmath}
\usepackage{amssymb}
\usepackage{amscd}
\usepackage{amsthm}

\usepackage{dblaccnt}
\usepackage{accents}
\usepackage{graphicx}
\usepackage{array}

\newtheorem{theorem}{Theorem}

\newtheorem{lemma}[theorem]{Lemma}
\newtheorem{corollary}[theorem]{Corollary}

\newtheorem{remark}[theorem]{Remark}
\newtheorem{assumption}[theorem]{Assumption}

\newcommand*{\N}{\ensuremath{\mathbb{N}}}
\newcommand*{\Z}{\ensuremath{\mathbb{Z}}}
\newcommand*{\R}{\ensuremath{\mathbb{R}}}
\newcommand*{\C}{\ensuremath{\mathbb{C}}}
\renewcommand{\i}{\mathrm{i}}
\renewcommand{\phi}{\varphi}
\renewcommand{\rho}{{\varrho}}
\renewcommand{\epsilon}{{\varepsilon}}

\renewcommand{\d}[1]{\,\mathrm{d}#1 \,}

\newcommand{\J}{\mathcal{J}} 
\newcommand{\0}{{0}} 

\newcommand{\I}{{\mathcal{I}}}
\newcommand{\B}{{\mathcal{B}}}

\newcommand{\A}{{\mathcal{A}}}
\renewcommand{\I}{{\mathcal{I}}}

\renewcommand{\S}{\mathcal{S}}
\usepackage[latin1]{inputenc}
\usepackage{tikz-cd}

\newcommand{\grad}{\nabla}

\newcommand{\W}{{W_{\hspace*{-1pt}{\Lambda}}}} 
\newcommand{\Wast}{{W_{\hspace*{-1pt}{\Lambda}^\ast}}} 

\newlength{\dhatheight}

\usepackage{color}

\begin{document}

\sloppy

\title{Scattering problems from slightly perturbed periodic surfaces: Part II. High order numerical method}
\author{Ruming Zhang\thanks{Center for Industrial Mathematics, University of Bremen
; \texttt{rzhang@uni-bremen.de}}}
\date{}
\maketitle

\begin{abstract}
In this paper, we develop a high order numerical method for the numerical solutions of scattering problems with slightly perturbed periodic surfaces in two dimensional spaces. Based on the regularity property introduced in Part I, the decaying rate of the incident field could be transferred directly to the  total field for small perturbations. Thus the finite section method could reach a high accuracy rate. With the help of a  modification of the truncated problem, the problem is solved by a finite element method. The convergence of the finite element method is proved and numerical examples have been carried out to show the efficiency of the numerical scheme.
\end{abstract}

\section{Introduction}

Numerical simulations of scattering problems with rough surfaces are always challenging. The unbounded domain always needs to be truncated, and the error caused by the truncation depends greatly on the properties of the total fields. 
 This paper considers a special family of rough surface scattering problems, i.e., when the incident field satisfies certain conditions and the rough surface is a slightly perturbation of a periodic one.

Generally speaking, the slightly perturbed periodic surfaces could be treated as rough surfaces if the periodicities are ignored, thus there are some known method for the theoretical and numerical analysis of these problems. We refer to \cite{Chand1996c,Chand1999,Zhang2003,Chand2006a} for the integral equation method. An alternative way is to apply the variational method, see \cite{Chand2005,Chand2010}. The numerical methods for the rough surface scattering problems are always based on the decaying property of the total field (see \cite{Meier2000,Meier2001} for the numerical solutions of the integral equations and \cite{Chand2010} for the finite section method). Due to the limited decaying rate, the finite section method always converges slowly.

Until recent years, theoretical and numerical analysis based on the Floquet-Bloch transform provide another way to deal with these problems. With the help of this method, a couple of scattering problems with (locally perturbed) periodic background have been studied and efficient numerical methods have been developed. We would like to mention \cite{Coatl2012,Hadda2016} for the scattering problems from locally perturbed periodic media. With a domain transformation technique, the Bloch transform was applied to the scattering problems with locally perturbed periodic surfaces in \cite{Lechl2016}. Based on that, numerical methods have been developed, for 2D cases see \cite{Lechl2016a,Lechl2017} and for 3D cases see \cite{Lechl2016b}. The method has also been extended to globally perturbed problems, see \cite{Zhang2018c,Zhang2018d}. On the other hand, a high order numerical method has also been developed, see \cite{Zhang2017e}. It was also shown that, the smoothness of the Bloch transformed field with respect to the quasi-periodicity parameter depends on the incident field, except for the square-root like singularities brought by the Dirichlet-to-Neumann map.

The paper \cite{Zhang2018e} considers the scattering problems with globally perturbed periodic surfaces. It is proved that, when the incident field satisfies special conditions and the perturbation of the periodic surface is small enough, the regularity of the Bloch transformed field with respect to the quasi-periodicity parameter only depends on the incident field. This result provides a possibility to develop an efficient numerical method when the incident field is smooth enough. The problem is modified into an equivalent one due to the singularity of the Dirichlet-to-Neumann map, and the well-posedness of the new problem comes directly from the equivalence. The next step is to truncate the term defined in the unbounded domain. A high order convergence rate comes from the high regularity of the Bloch transformed field. Then a classical finite element method is adopted, and the convergence is proved for the numerical method.

The rest of this paper is organized as follows. In Section 2, we briefly recall the mathematical modal and well-posedness of the rough surface scattering problems, and the properties of the Bloch transformed problems. The third section recalls the regularity result from \cite{Zhang2018e}. In the fourth section, we show the classic finite element method introduced in \cite{Lechl2017} and its disadvantages. The fifth section shows the modified variational problems, with the help of change of variables. In the sixth section, we apply the finite section method for the truncation of the unbounded term, and estimate the error of the truncation. In the seventh section, we apply the finite element method to the truncation of the modified variational problem. In the last section, the numerical experiments are carried out to show the efficiency of the numerical scheme.

\section{Scattering problems from slightly perturbed periodic surfaces}

\subsection{Mathematical model and well-posedness}

In this section, the mathematical formulation of the scattering problems in two dimensional spaces is presented. For details we refer to \cite{Chand2005,Chand2010}.

\begin{figure}[H]
\centering
\includegraphics[width=15cm]{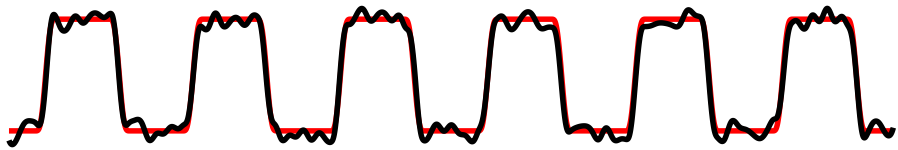}
\caption{Red curve: periodic surface; black curve: slightly perturbed periodic surface.}
\end{figure}

Let $\zeta$ be a bounded function defined in $\R$, and the surface $\Gamma\subset\R^2$ is defined by
\begin{equation*}
\Gamma:=\big\{x\in\R^2:\, x_2=\zeta(x_1)\text{ where }x_1\in\R\big\}.
\end{equation*}
The domain above the surface is defined by
\begin{equation*}
\Omega:=\left\{x\in\R^2:\,x_2>\zeta(x_1)\right\}.
\end{equation*} 
Suppose $H$ is a constant satisfies $H>\max_{t\in\R}\left\{\zeta(t)\right\}$, then $\Gamma_H:=\R\times\{H\}$ is a straight line lies above $\Gamma$. Define the domain above $\Gamma$. Define the domain with finite height:
\begin{equation*}
\Omega_H:=\left\{x\in\R^2:\,\zeta(x_1)<x_2<H\right\}.
\end{equation*}

We consider the following equations in $\Omega_H$:
\begin{eqnarray}
&& \Delta u+k^2 u=0\quad\text{ in }\Omega_H;\\
&& u=0\quad\text{ on }\Gamma;\\
&&\frac{\partial u}{\partial x_d}(\widetilde{x},H)=T^+\left[u\big|_{\Gamma_H}\right]+f\quad\text{ on }\Gamma_H,\label{eq:boundary_condition}
\end{eqnarray}
where $T^+$ is the Dirichlet-to-Neumann map defined by
\begin{equation}\label{eq:DtN}
T^+\phi=\frac{\i}{{2\pi}}\int_{\R} \sqrt{k^2-|{\xi}|^s}e^{\i  x_1\cdot{\xi}}\widehat{\phi}({\xi})\d{\xi}\quad\text{ for }\phi=\frac{1}{2\pi}\int_{\R} e^{\i x_1\cdot{\xi}}\widehat{\phi}({\xi})\d{\xi},
\end{equation}
and $f$ is defined by
\begin{equation}\label{eq:def_f}
f:=\frac{\partial u^i}{\partial x_2}(x_1,H)-T^+\left[u^i|_{\Gamma_H}\right]
\end{equation}
for the incident field $u^i$ that satisfies the Helmholtz equation in $\Omega$.

The weak formulation for the scattering problem is, given any $f$ in the weighted Sobolev space $H_r^{-1/2}(\Gamma_H)$, to find a solution $u\in\widetilde{H}_r^1(\Omega_H)$ such that
\begin{equation}\label{eq:var_origional}
\int_{\Omega_H}\left[\nabla u\cdot\nabla\overline{v}-k^2u\overline{v}\right]\d x-\int_{\Gamma_H}T^+\left[u|_{\Gamma_H}\right]\overline{v}\d s=\int_{\Gamma_H}f\overline{v}\d s,
\end{equation}
for all $v\in\widetilde{H}^1_r(\Omega_H)$ with compact support in $\overline{\Omega_H}$.
\begin{remark}
The tilde in $\widetilde{H}_r^1(\Omega_H)$ shows that the functions in this space belong to $H_r^1(\Omega_H)$ and satisfy homogeneous Dirichlet boundary condition on $\Gamma$. Similar notations are utilized for other spaces, e.g., $H_0^r(\Wast;\widetilde{H}_{\alpha}^s(\Omega^\Lambda_H))$.
\end{remark}

From \cite{Chand2010}, the unique solvability of the variational problem \ref{eq:var_origional} has been proved in weighted Sobolev spaces.
\begin{theorem}\label{th:solv}
If $\Gamma$ is Lipschitz continuous, $f\in H_r^{-1/2}(\Gamma_H)$ for $|r|<1$, then there is a unique solution $u\in\widetilde{H}_r^1(\Omega_H)$ for the variational problem \eqref{eq:var_origional}.
\end{theorem}

\subsection{The Bloch transformed problem}

In this subsection, we  apply the Bloch transform to  the scattering problems. Suppose $\Lambda>0$ and the surface $\zeta$ is $\Lambda$-periodic, and let the function $\zeta_p$ be a global perturbation of $\zeta$. Let the surface defined by $\zeta_p$ be denoted by $\Gamma_p$, and the domains $\Omega^p$ and $\Omega^p_H$ are defined in the same way.

 Let $\Lambda^*:=2\pi/\Lambda$, define the periodic cell and dual cell by
\begin{equation*}
\W:=\left(-\frac{\Lambda}{2},\frac{\Lambda}{2}\right];\quad
\Wast:=\left(-\frac{\Lambda^*}{2},\frac{\Lambda^*}{2}\right]=\left(-\frac{\pi}{\Lambda},\frac{\pi}{\Lambda}\right].
\end{equation*}
Now we can define $\Gamma^\Lambda_H$ and $\Omega^\Lambda_H$, which are  $\Gamma_H$ and $\Omega_H$ restricted in one periodic cell $\W\times\R$, i.e.,
\begin{equation*}
\Gamma^\Lambda_H=\Gamma_H\cap\left[\W\times\R\right],\quad \Omega^\Lambda_H=\Omega_H\cap\left[\W\times\R\right].
\end{equation*} 

 Let $\Theta_p$ be a diffeomorphism that maps $\Omega^p_{H_0}$ to $\Omega_{H_0}$ for some $\|\zeta\|_\infty<H_0<H$, and extend $\Theta_p$ by identity in $\R\times[H_0,\infty)$. Thus the support of $\Theta_p-I$ is contained in $\Omega_{H_0}$. Let the transformed total field $u_T:=u\,\circ\,\Theta_p$,  then by direct calculation, $u_T\in\widetilde{H}_r^1(\Omega_H)$   satisfies the following variational problem in the periodic domain $\Omega_H$
\begin{equation}\label{eq:var_T}
\int_{\Omega_H}\left[A_p\nabla u_T\cdot\nabla\overline{v_T}-k^2c_p u_T\overline{v_T}\right]\d { x}-\int_{\Gamma_H}T^+\left[u_T|_{\Gamma_H}\right]\overline{v_T}\d s=\int_{\Gamma_H}f\overline{v_T}\d s,
\end{equation}
for all $v_T:=v\circ \Theta\in\widetilde{H}^1(\Omega_H)$,  where
\begin{eqnarray*}
&& A_p({x}):=\left|\det\grad\Theta_p({ x})\right|\left[\left(\grad\Theta_p({ x})\right)^{-1}\left(\grad\Theta_p({ x})\right)^{-T}\right]\in L^\infty\left(\Omega_H,\R^{2\times 2}\right);\\
&& c_p(x):=\left|\det\grad\Theta_p({ x})\right|\in L^\infty(\Omega_H).
\end{eqnarray*}
Thus the supports of both $A_p-I_2$ and $c_p-1$ are subsets of $\Omega_{H_0}$.

Apply the the Bloch transform to \eqref{eq:var_T}, we arrive at the variational problem for $w=\J_{\Omega_H} u_T$ with test function $z={\J_{\Omega_H} {v_T}}$:
\begin{equation}\label{eq:var_Bloch}
\int_\Wast a_\alpha(w(\alpha,\cdot),z(\alpha,\cdot))\d\alpha+\, b(w,z)=\int_\Wast\int_{\Gamma^\Lambda_H}F(\alpha,x) \overline{z(\alpha,x)}\d{ s(x)}\d\alpha,
\end{equation}
where
\begin{eqnarray*}
 &&a_{\alpha}(w,z):=\int_{\Omega^\Lambda_H}\left[\nabla w\cdot\nabla \overline{z}-k^2 w\overline{z}\right]\d{ x}-\int_{\Gamma^\Lambda_H}T^+_{\alpha}(w)\overline{z}\d s,\\
 &&\begin{aligned}
b(w,z)&=\int_{\Omega_H}\Big[ (A_p-I_2)\grad( \J_{\Omega_H}^{-1} w)\cdot\grad\overline{\left(\J_{\Omega_H}^{-1} {z}\right)}-k^2 (c_p-1)(\J_{\Omega_H}^{-1}w)\cdot\overline{\left(\J_{\Omega_H}^{-1} {z}\right)}\Big]\d x,\\
&=\int_\Wast\int_{\Omega^\Lambda_H}\left[\J_{\Omega_H}\left[(A_p-I_2)\grad( \J_{\Omega_H}^{-1} w)\cdot\grad\overline{z}-k^2\J_{\Omega_H}\left[(c_p-1)( \J_{\Omega_H}^{-1} w)\right]\overline{z}\right]\right]\d x\d\alpha,
\end{aligned}\\
&& F(\alpha,x)=\left(\J_{\Gamma_H} f\right)(\alpha,x).
\end{eqnarray*}
Moreover, the operator $T^+_{\alpha}$ is the well-known ${\alpha}$-quasi-periodic Dirichlet-to-Neumann operator from $H^{1/2}_\alpha(\Gamma_H)$ to $H^{-1/2}_\alpha(\Gamma_H)$ defined by
\begin{equation*}
T_{\alpha}^+\phi=\i\sum_{{ j}\in\Z^{d-1}}\sqrt{k^2-|\Lambda^* j-{\alpha}|^2}\widehat{\phi}({ j})e^{\i(\Lambda^*{ j}-{\alpha})\cdot \widetilde{x}}\quad\text{ for }\phi=\sum_{{ j}\in\Z^{d-1}}\widehat{\phi}({ j})e^{\i(\Lambda^*{ j}-{\alpha})\cdot \widetilde{x}}.
\end{equation*}

The equivalence, well-posedness and regularity results are easily extended from locally perturbed cases (see \cite{Lechl2016,Lechl2017}), and the results have been proved in \cite{Zhang2018c}, and we list these results in the following theorem.

\begin{theorem}[Lemma 7, Theorem 8-10, \cite{Zhang2018c}]\label{th:well-posed}
Assume that $f\in H_r^{-1/2}(\Gamma_H)$ for $r\in(0,1)$ and $\zeta,\,\zeta_p$ are Lipschitz continuous functions.
\begin{enumerate}
\item $u\in \widetilde{H}^1_r(\Omega_H)$ satisfies \eqref{eq:var_origional} if and only if $w=\J_{\Omega_H}u_T\in H_0^r(\Wast;\widetilde{H}^1_\alpha(\Omega^\Lambda_H))$ satisfies \eqref{eq:var_Bloch} with $F(\alpha,\cdot)=\J_{\Omega_H}f$.
\item Given  any $F\in H_0^r(\Wast;H_\alpha^{-1/2}(\Gamma_H^\Lambda))$ for some $r\in[0,1)$,  the variational problem \eqref{eq:var_Bloch} has a unique solution in $H_0^r(\Wast;\widetilde{H}^1_\alpha(\Omega^\Lambda_H))$.
\item If $f\in H_r^{1/2}(\Gamma_H)$ and $\zeta\in C^{2,1}(\R)$, the solution $w\in H_0^r(\Wast;\widetilde{H}^2_\alpha(\Omega^\Lambda_H))$.
\item If $r\in(1/2,1)$, then the solution $w\in H_0^r(\Wast;\widetilde{H}^1_\alpha(\Omega^\Lambda_H))$ equivalently satisfies for all $\alpha\in\Wast$ and $z_\alpha\in\widetilde{H}_\alpha^1(\Omega^\Lambda_H)$ such that
\begin{equation}\label{eq:var_Bloch_continuous}
a_\alpha(w(\alpha,\cdot),z_\alpha)+ b(w(\alpha,\cdot),z_\alpha)=\int_{\Gamma^\Lambda_H}F(\alpha,x)\overline{z_\alpha(x)}\d s(x).
\end{equation}
\end{enumerate}
 
\end{theorem}

\section{Regularity  in Sobolev spaces}

\subsection{High order regularity}
In this section, we conclude some regularity properties of the Bloch transformed scattering problems from locally or globally perturbed periodic surfaces obtained in \cite{Zhang2018e}. The result from \cite{Chand2010}, only for $|r|<1$, the decaying rate of the incident field could be transferred to the total field. However, for some special cases, the result holds for $r\geq 1$ (see \cite{Zhang2018e}). In the following, we will consider these special cases and list the known results obtained in \cite{Zhang2017e,Zhang2018e}.

For convenience, we define the discrete set that depends on $k$ and $\Lambda^*$:
\begin{equation}
\S:=\big\{\alpha_0\in\overline{\Wast}:\,\exists\, j\in\Z \text{ such that }|\Lambda^*j-\alpha_0|=k\big\}.
\end{equation} 
It is a non-empty discrete set that has at most three points. From the analysis in \cite{Zhang2017e}, there are two different kinds of wave numbers that correspond to two kinds of representations of $\S$. Define the  real number $\underline{k}$ by
\begin{equation}
\underline{k}:=\min\{|\Lambda^*j-k|:\,j\in\Z\},
\end{equation}
then $0\leq\underline{k}\leq\Lambda^*/2$ and the points in $\S$ could be represented by $\underline{k}$:
\begin{itemize}
\item Case 1, $\underline{k}=m\Lambda^*/2$ for some $m=0,1$, then $\S=\{\underline{k}+\Lambda^*j:\,j\in\Z\}\cap\overline{\Wast}$;
\item Case 2, $\underline{k}\neq m\Lambda^*/2$ for any $m=0,1$, then $\S=\{\underline{k}+\Lambda^*j,\,-\underline{k}+\Lambda^*j:\,j\in\Z\}\cap\overline{\Wast}$.
\end{itemize}

As was shown in \cite{Zhang2018e}, the points in $\S$ are critical, as the $\alpha$-dependent quasi-periodic Dirichlet-to-Neumann map $T^+_\alpha$ has a square-root like singularity in the neighbourhood of these points.   

In \cite{Zhang2018e}, when the right hand side belongs to a closed subspace of $H_0^n(\Wast;\widetilde{H}^1_\alpha(\Gamma^\Lambda_H))$ for some $n\in\N$, the regularity of the solutions of \eqref{eq:var_Bloch} with respect to $\alpha$ only depends on $r$.

\begin{figure}[H]
\centering
\includegraphics[width=15cm]{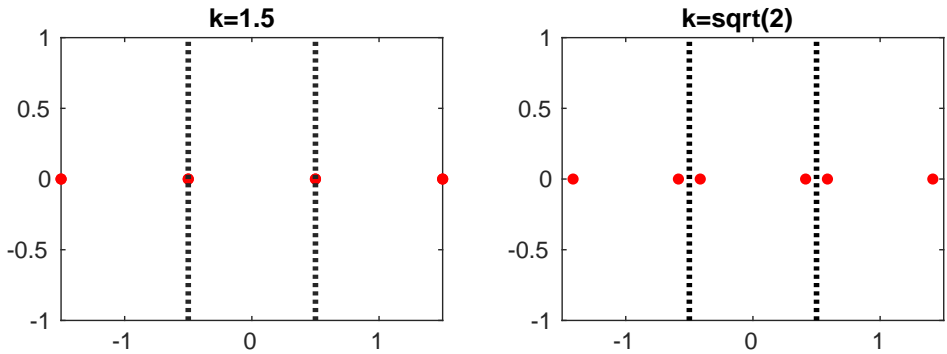}
\caption{Distribution of the set $\S$ for $\Lambda=2\pi$. Left: $k=1.5$; right: $k=\sqrt{2}$. $\Wast$ is the line $\R\times\{0\}$ between the black dotted lines.}
\end{figure}


The first theorem considers Bloch transformed problem \eqref{eq:var_Bloch}, when the perturbation of the periodic is small enough. First, we introduce the space $H_{00}^n(\Wast;\S;\widetilde{H}^1_\alpha(\Omega^\Lambda_H))$ by the closure of the set
\begin{equation}
\left\{\phi\in C^\infty(\Wast\times\Omega^\Lambda_H):\,\exists\delta>0,s.t.,\phi(\alpha,\cdot)=0\text{ for }\alpha\in \bigcup_{\alpha_0\in\S}B(\alpha_0,\delta)\cap\Wast\right\}
\end{equation}
with respect to the $H_0^n(\Wast;H^s_\alpha(\Omega^\Lambda_H))$-norm, where $B(\alpha_0,\delta)=(\alpha_0-\delta,\alpha_0+\delta)$. 

\begin{theorem}[Theorem 24, \cite{Zhang2018e}]
Suppose $\zeta,\,\zeta_p$ are Lipschitz continuous functions.  $\zeta$ is $\Lambda$-periodic and $\zeta_p$ is a slight perturbation of $\zeta$.  If $F\in H_{00}^n(\Wast;\S;H^{-1/2}_\alpha(\Gamma^\Lambda_H))$ for some $n\in\N$, then $w\in H_0^n(\Wast;\widetilde{H}^{1}_\alpha(\Omega^\Lambda_H))$.
\label{th:well-posed_high_order}
\end{theorem}

 For simplicity, we adopt notations from \cite{Zhang2018e}. From Riesz representation theorem, there are operators $\A,\,\B_p\in\mathcal{L}(L^2(\Wast;\widetilde{H}^1_\alpha(\Omega^\Lambda_H)))$ that satisfies 
\begin{equation}
\left<\A w,z\right>=\int_\Wast a_\alpha(w(\alpha,\cdot),z(\alpha,\cdot))\d\alpha;\quad\left<\B_p w,z\right>=b(w,z)
\end{equation}
for any $w,z\in L^2(\Wast;\widetilde{H}^1_\alpha(\Omega^\Lambda_H))$, where $\left<\cdot,\cdot\right>$ is the inner product in the space $L^2(\Wast;\widetilde{H}^1_\alpha(\Omega^\Lambda_H))$. Use Riesz  representation theorem again, there is a $G\in H_{00}^n(\Wast;\S;\widetilde{H}^1_\alpha(\Omega^\Lambda_H))$ such that
\begin{equation}
\int_\Wast\int_{\Gamma^\Lambda_H}F(\alpha,x)\overline{z(\alpha,x)}\d x\d \alpha=\left<G,z\right>.
\end{equation}
 Theorem 24, \cite{Zhang2018e}, that $\A+\B_p$ is invertible in $H_{00}^n(\Wast;\S;\widetilde{H}^1_\alpha(\Omega^\Lambda_H))$, and for any $G\in H_{00}^n(\Wast;\S;\widetilde{H}^1_\alpha(\Omega^\Lambda_H))$, $(\A+\B_p)^{-1}G\in H_0^n(\Wast;\widetilde{H}^1_\alpha(\Omega^\Lambda_H))$.
 
 For a small enough $\epsilon>0$, let $\mathcal{X}_\epsilon$ be a smooth cut-off function satisfies
 \begin{equation}
 \mathcal{X}_\epsilon(t)=\begin{cases}
 1,\quad |t|<\epsilon/2;\\
 0,\quad|t|>1;\\
 \text{smooth}, \quad\text{otherwise}.
 \end{cases}
 \end{equation}
Define the modified Dirichlet-to-Neumann map $T_\alpha^\epsilon$ by
\begin{equation}
T_{\alpha}^\epsilon\phi=\i\sum_{{ j}\in\Z^{d-1}}\mathcal{X}_\epsilon\left(k-|\Lambda^*j-\alpha|\right)\sqrt{k^2-|\Lambda^* j-{\alpha}|^2}\widehat{\phi}({ j})e^{\i(\Lambda^*{ j}-{\alpha})\cdot \widetilde{x}}\,\,\text{ for }\phi=\sum_{{ j}\in\Z^{d-1}}\widehat{\phi}({ j})e^{\i(\Lambda^*{ j}-{\alpha})\cdot \widetilde{x}}.
\end{equation}
Replace the term $T^+_\alpha$ in $a_\alpha(w(\alpha,\cdot),z(\alpha,\cdot))$ by $T^\epsilon_\alpha$, then there is an operator $\A^\epsilon\in\mathcal{L}(L^2(\Wast;\widetilde{H}^1_\alpha(\Omega^\Lambda_H)))$ such that
\begin{equation}
\begin{aligned}
\left<\A^\epsilon w,z\right>=\int_\Wast &\left[\int_{\Omega^\Lambda_H}\left(\nabla w(\alpha,\cdot)\cdot\nabla \overline{z(\alpha,\cdot)}-k^2 w(\alpha,\cdot)\overline{z(\alpha,\cdot)}\right)\d{ x}\right.\\
&\qquad\qquad\qquad\qquad\qquad\quad\left.-\int_{\Gamma^\Lambda_H}T^\epsilon_\alpha(w(\alpha,\cdot))\overline{z(\alpha,\cdot)}\d s\right]\d\alpha.
\end{aligned}
\end{equation} 
Thus $\A^\epsilon$ is also a bounded linear operator in $\mathcal{L}(H_0^n(\Wast;\widetilde{H}^1_\alpha(\Omega^\Lambda_H)))$ for any $n\in\N$. 
From \cite{Zhang2018e} it was proved that, for small enough $\epsilon>0$, 
\begin{equation}\label{eq:equivalence}
w=\left(\A+\B_p\right)^{-1}G=\left(\A^\epsilon+\B_p\right)^{-1}G.
\end{equation}

From Theorem \ref{th:well-posed}, the original problem \eqref{eq:var_origional} and the Bloch transformed problem \eqref{eq:var_Bloch} are equivalent. Thus we can get the following result.

\begin{theorem}[Theorem 25, \cite{Zhang2018e}]
Suppose $\zeta,\,\zeta_p$ are Lipschitz continuous functions and $\zeta$ is $\Lambda$-periodic. $\zeta_p$ is a slight perturbation of $\zeta$. Given an incident field $u^i\in H_n^1(\Omega_H^p)$ for some $n\in\N$. If  $f$, which is defined by \eqref{eq:def_f}, satisfies that $\left(\J_{\Lambda_H}f\right)(\alpha,\cdot)=0$ for $\alpha$ in a neighbourhood of $\S$, then there is a unique solution $u\in \widetilde{H}_n^1(\Omega_H^p)$ for the variational problem \eqref{eq:var_origional}.
\label{th:decay}
\end{theorem}

For a special case, i.e., when $\zeta$ is a constant function, the period $\Lambda$ could be chosen to be any positive number. Let $\J_{\Omega_H}(\Lambda)$ be the Bloch transform defined with the period $\Lambda$, we have the following result.

\begin{corollary}[Corollary 26, \cite{Zhang2018e}]
Suppose there is an $h_0\in\R$ such that $\|\zeta_p-h_0\|_{W^{1,\infty}(\R)}$ is small enough. Given an incident field $u^i\in H_n^1(\Omega_H^p)$ and define $f$ by \eqref{eq:boundary_condition}. Suppose there is a $\Lambda>0$ such that  $\left(\J_{\Gamma_H}(\Lambda)f\right)\in H_{00}^n(\Wast;\S;H^{-1/2}_\alpha(\Gamma^\Lambda_H))$, then there is a unique solution $u\in \widetilde{H}_n^1(\Omega_H^p)$ for the variational problem \eqref{eq:var_origional}.
\end{corollary}

\subsection{Continuous properties of solutions}

Suppose $W\subset\R^2$ is any bounded domain and $S(W)$ is a Sobolev space of functions defined in $W$, where $S(W)=H^m(W)$ with $m\in\N$ is a fixed integer. Let $\I\subset\R$ be any finite interval. Let the space $H^n(\I;S(W))$ be defined as
\begin{equation}
H^n(\I;S(W)):=\left\{\phi\in\mathcal{D}'(\I\times W):\, \int_\I\left\|\frac{\partial^\ell }{\partial\alpha^\ell}\phi(\alpha,\cdot)\right\|_{S(W)}\d\alpha\right\}.
\end{equation} 

\begin{remark}
In this paper, the space $H_0^n(\I;S(W))$ the subsection that functions are periodic with respect to the first variable, the norm is the same as that of $H^n(\I;S(W))$. However, for spaces $H_0^n(\I;H^s_{\alpha}(W))$ or $H_0^n(\I;H^s_{g(t)}(W))$, the norms may be defined in different ways.
\end{remark}

Then we can define the $C^n(\I;S(W))$ norm for any $n\in\N$ by
\begin{equation}
\|\phi\|_{C^n(\I;S(W))}:=\sum_{j=0}^n\left[\sup_{\alpha\in\I}\left\|\frac{\partial^j \phi(\alpha,\cdot)}{\partial\alpha^j}\right\|_{S(W)}\right].
\end{equation}
Let the H\"{o}lder coefficient with respect to $\alpha$ be
\begin{equation}
|\phi|_{C^{0,\gamma}(\I;S(W))}:=\sup_{\alpha_1\neq\alpha_2}\frac{\|\phi(\alpha_1,\cdot)-\phi(\alpha_1,\cdot)\|_{S(W)}}{|\alpha_1-\alpha_2|^\gamma}.
\end{equation}
Then define the  $C^{n,\gamma}(\I;S(W))$ norm  by
\begin{equation}
\|\phi\|_{C^{n,\gamma}(\I;S(W))}:=\|\phi\|_{C^n(\I;S(W))}+\left|\frac{\partial^n}{\partial\alpha^n}\phi(\alpha,\cdot)\right|_{C^{0,\gamma}(\I;S(W))}
\end{equation}

The well-known Sobolev embedding theorem could easily be extended to the space $H_0^n(\I;S(W))$. First we recall a classical  Minkowski integral inequality, see [\cite{Hardy1988}, Theorem 202].

\begin{lemma}\label{th:mink}
Suppose $(S_1,\mu_1)$ and $(S_2,\mu_2)$ are two measure spaces and $F:\,S_1\times S_2\rightarrow \R$ is measuable. Then the following inequality holds for any $p\geq 1$.
\begin{equation}
\left[\int_{S_2}\left|\int_{S_1}F(y,z)\d \mu_1(y)\right|^p\d\mu_2(z)\right]^{1/p}\leq\int_{S_1}\left(\int_{S_2}|F(y,z)|^p\d \mu_2(z)\right)^{1/p}\d\mu_1(y).
\end{equation}
\end{lemma}

 We describe the result for one dimensional subspace $\I$.

\begin{lemma}\label{th:embedding}
For any $n\in\N_+$ and $m\in\N$, the following continuous embedding holds.
\begin{equation}
H^n(\I;S(W))\subset C^{n-1,1/2}(\I;H^m(W)).
\end{equation}
\end{lemma}

\begin{proof}
We only need the proof for $n=1$. First let $m=0$. For any $\phi\in C^\infty(\I\times\W)$ and periodic with respect to $\alpha$, let $\I=(a_0,A_1)$ then $\phi(a_0,\cdot)=\phi(A_1,\cdot)$ in $S(W)$. Then
\begin{equation}
\phi(\alpha,x)=\phi(a_0,x)+\int_{a_0}^\alpha \frac{\partial\phi(\alpha,x)}{\partial\alpha}\d\alpha.
\end{equation}
For $\alpha_1<\alpha_2$ belong to $\I$,
\begin{equation}
\phi(\alpha_2,x)-\phi(\alpha_1,x)=\int_{\alpha_1}^{\alpha_2}\frac{\partial\phi(\alpha,x)}{\partial\alpha}\d\alpha.
\end{equation}
Then from Minkowski integral inequality described in Lemma \ref{th:mink},
\begin{equation}
\begin{aligned}
\left\|\phi(\alpha_2,\cdot)-\phi(\alpha_1,\cdot)\right\|_{L^2(S(W))}&=\left(\int_W \left|\phi(\alpha_2,x)-\phi(\alpha_1,x)\right|^2\d x\right)^{1/2}\\
&=\left(\int_W\left|\int_{\alpha_1}^{\alpha_2}\frac{\partial\phi(\alpha,x)}{\partial\alpha} \d\alpha\right|^2\d x\right)^{1/2}\\
&\leq \int_{\alpha_1}^{\alpha_2}\left(\int_W\left|\frac{\partial\phi(\alpha,x)}{\partial\alpha}\right|^2\d x\right)^{1/2}\d\alpha.
\end{aligned}
\end{equation}
Use the Cauchy-Schwartz inequality,
\begin{equation}
\begin{aligned}
\left\|\phi(\alpha_2,\cdot)-\phi(\alpha_1,\cdot)\right\|_{L^2(S(W))}&\leq \left(\int_{\alpha_1}^{\alpha_2}1\d\alpha\right)^{1/2}\left(\int_{\alpha_1}^{\alpha_2}\int_W\left|\frac{\partial\phi(\alpha,x)}{\partial\alpha}\right|^2\d x\d\alpha\right)^{1/2}\\
&\leq |\alpha_2-\alpha_1|^{1/2}\left(\int_{\I}\int_W\left|\frac{\partial\phi(\alpha,x)}{\partial\alpha}\right|^2\d x\d\alpha\right)^{1/2}\\
&\leq |\alpha_2-\alpha_1|^{1/2}\|\phi\|_{H^1(\I;L^2(W))}.
\end{aligned}
\end{equation}
From the density of $C^\infty(\I\times W)$ in $H^1(\I;L^2(W))$, it is easy to obtain the inequality for $\phi\in H^1(\I;L^2(W))$. The case that $m>0$ could be investigated in the same way.  The proof is finished.
\end{proof}

\begin{remark}
More generalized Sobolev embedding results could also be extended, the basic idea is similar to the proof here, i.e., to apply the H\"{o}lder inequality and Minkowski integral inequality. 
\end{remark}

From Sobolev embedding theorem, the continuity of the solution is easily obtained.

\begin{corollary}
When $\zeta_p$ is a small enough perturbation of $\zeta$ and $F\in H_{00}^{n+1}(\Wast;\S;H^{-1/2}_\alpha(\Gamma^\Lambda_H))$ for some $n\in\N$, then the unique solution $w\in C_0^{n,1/2}(\Wast;\widetilde{H}^1_\alpha(\Omega^\Lambda_H))$.
\end{corollary}

\section{Classical numerical method and its disadvantages}

In this section, we recall the numerical scheme introduced in \cite{Lechl2017} by Lechleiter \& Zhang and show that this method fails for higher regularity  problems. 

\subsection{Interpolation of the quasi-periodic parameter}

Let the interval $\I=(A_0,A_1)$ be divided uniformly into $N$ subintervals. Let the grid points be
\begin{equation}
\alpha_N^{(j)}=A_0+\frac{A_1-A_0}{N}j\quad\text{for }j=1,2,\dots,N.
\end{equation}
The basis $\left\{\psi_{N}^{(j)}\right\}_{j=1}^{N}$ are defined as 
\begin{equation}
\psi_{N}^{(j)}=\frac{1}{N}\sum_{\ell=-N/2+1}^{N/2}\exp\left(\i\ell\left[t-\alpha_N^{(j)}\right]\frac{2\pi}{A_1-A_0}\right),\quad j=1,2,\dots,N.
\end{equation}
It is obvious that the basic functions satisfies the following properties.
\begin{itemize}
\item $\phi_{2N}^{j}=\delta_{j,j'}$, where $\delta_{j,j}=1$, otherwise $\delta_{j,j'}=0$.
\item The functions $\left\{\psi_{2N}^{(j)}\right\}_{j=1}^{2N}$ are orthogonal. Moreover,
\begin{equation}
\int_\Wast \psi_{N}^{(j)}(t)\overline{\psi_{N}^{(j')}(t)}\d t=\frac{A_1-A_0}{N}\delta_{j,j'}.
\end{equation}
\end{itemize}

We consider the interpolation of a function $\phi\in C_0^{n,\gamma}(\I;S(W))$ where $0<\gamma<1$ and $m\in\N$ with the basis $\left\{\psi_{N}^{(j)}\right\}_{j=1}^{N}$. Then the interpolation of $\phi$ is represented by the basis in the form of: 
\begin{equation}
\phi_N(\alpha,x):=\sum_{j=-N/2+1}^{N/2}\psi_N^{(j)}(\alpha)\phi\left(\alpha_N^{(j)},x\right).
\end{equation}
To investigate the error estimation of the interpolation, we have to consider the Fourier coefficients first.

\begin{lemma}\label{th:decay_fourier}
Suppose $\phi\in C_0^{n,\gamma}(\I;S(W))$. Let $v$ have the Fourier series with respect to $\alpha$ 
\begin{equation}
\phi(\alpha,x)=\sum_{j\in\Z}\widehat{\phi}_j(x)\exp\left(\i j\alpha\frac{2\pi}{A_1-A_0}\right),
\end{equation}
where $\widehat{\phi}_j\in S(W)$. Moreover, for any $j\in \Z\setminus\{0\}$,
\begin{equation}
\|\widehat{\phi}_j\|_{S(W)}\leq C\left|\frac{A_1-A_0}{2\pi j}\right|^{n+\gamma}\left\|\phi\right\|_{C_0^{n,\gamma}(\I;L^2(W))}.
\end{equation}
\end{lemma}

\begin{proof}
First consider $m=0$. As $\phi\in C_0^{n,\gamma}(\I;L^2(W))$,
\begin{equation}
\|\phi\|_{C_0^{n,\gamma}(\I;L^2(W))}=\sum_{j=0}^n\left[\sup_{\alpha\in\I}\left\|\frac{\partial^j \phi(\alpha,\cdot)}{\partial\alpha^j}\right\|_{L^2(W)}\right]+\sup_{\alpha_1\neq\alpha_2}\frac{\|\partial^n_\alpha\phi(\alpha_1,\cdot)-\partial^n_\alpha\phi(\alpha_1,\cdot)\|_{L^2(W)}}{|\alpha_1-\alpha_2|^\gamma}<\infty.
\end{equation}
Then by integration by parts, for $j\neq 0$,
\begin{equation}
\begin{aligned}
\widehat{\phi}_j(x)&=\frac{1}{A_1-A_0}\int_\I \phi(\alpha,x)e^{-\i j\alpha\frac{2\pi}{A_1-A_0}}\d \alpha\\
&=\frac{(A_1-A_0)^{n-1}}{(2\i\pi j)^n}\int_{\I}\frac{\partial^n}{\partial\alpha^n}\phi(\alpha,x)e^{-\i j\alpha\frac{2\pi}{A_1-A_0}}\d\alpha.
\end{aligned}
\end{equation}
From the periodicity, 
\begin{equation}
\begin{aligned}
\widehat{\phi}_j(x)&=\frac{(A_1-A_0)^{n-1}}{(2\i\pi j)^n}\int_\I \frac{\partial^n}{\partial\alpha^n}\phi\left(\alpha+\frac{A_1-A_0}{2j},x\right)e^{-\i j\left(\alpha+\frac{A_1-A_0}{2j}\right)\frac{2\pi}{A_1-A_0}}\d \alpha\\
&=-\frac{(A_1-A_0)^{n-1}}{(2\i\pi j)^n}\int_\I \frac{\partial^n}{\partial\alpha^n}\phi\left(\alpha+\frac{A_1-A_0}{2j},x\right)e^{-\i j\alpha\frac{2\pi}{A_1-A_0}}\d \alpha.
\end{aligned}
\end{equation}
Take the average of these two expressions,
\begin{equation}
\widehat{\phi}_j (x)=\frac{(A_1-A_0)^{n-1}}{2(2\i\pi j)^n}\int_\I\left[\frac{\partial^n}{\partial\alpha^n}\phi(\alpha,x)-\frac{\partial^n}{\partial\alpha^n}\phi\left(\alpha+\frac{A_1-A_0}{2j},x\right)\right]e^{-\i j\alpha\frac{2\pi}{A_1-A_0}}\d \alpha.
\end{equation}
Then by the Minkowski inequality in Lemma \ref{th:mink},
\begin{equation}
\begin{aligned}
&\left\|\widehat{\phi}_j\right\|_{L^2(W)}\\
=&\frac{|A_1-A_0|^{n-1}}{2(2\pi j)^{n}}\left[\int_W\left|\int_\I\left[\frac{\partial^n}{\partial\alpha^n}\phi(\alpha,x)-\frac{\partial^n}{\partial\alpha^n}\phi\left(\alpha+\frac{A_1-A_0}{2j},x\right)\right]e^{-\i j\alpha\frac{2\pi}{A_1-A_0}}\d \alpha\right|^2\d x\right]^{1/2}\\
\leq &\frac{|A_1-A_0|^{n-1}}{2(2\pi j)^{n}}\int_\I\left(\int_W\left\|\frac{\partial^n}{\partial\alpha^n}\phi(\alpha,x)-\frac{\partial^n}{\partial\alpha^n}\phi\left(\alpha+\frac{A_1-A_0}{2j},x\right)\right\|^2\d x\right)^{1/2} \d\alpha\\
\leq &\frac{|A_1-A_0|^{n-1}}{2(2\pi j)^{n}}\int_\I\left\|\frac{\partial^n}{\partial\alpha^n}\phi(\alpha,\cdot)-\frac{\partial^n}{\partial\alpha^n}\phi\left(\alpha+\frac{A_1-A_0}{2j},\cdot\right)\right\|_{L^2(W)} \d\alpha\\
\leq &\frac{|A_1-A_0|^{n-1}}{2(2\pi j)^{n}}\left|\frac{A_1-A_0}{2j}\right|^\gamma\left\|\partial^n_\alpha\phi\right\|_{C_0^{0,\gamma}(\I;L^2(W))}
\\\leq &C\left|\frac{A_1-A_0}{2\pi j}\right|^{n+\gamma}\left\|\phi\right\|_{C_0^{n,\gamma}(\I;L^2(W))}.
\end{aligned}
\end{equation}
The case that $m\geq 1$ could be proved in the same way. The proof is finished.
\end{proof}

From similar proof of Theorem 27 in \cite{Zhang2017e}, the error estimate of the interpolation of $\phi$ is estimated.

\begin{theorem}\label{th:interp_qper}
For $\phi\in C_0^{n,\gamma}(\I;S(W))$ for some $n\in\N$ and $\gamma\in (0,1)$, the difference between $\phi$ and $\phi_N$ is bounded by
\begin{equation}
\|\phi-\phi_N\|_{L^2(\I;H^m(W))}\leq CN^{-n+\gamma-1/2}\|\phi\|_{C_0^{n,\gamma}(\I;H^m(W))},
\end{equation}
where $C$ is a positive constant that does not depend on $N$.
\end{theorem}

\begin{proof}
First consider $m=0$. From the definition, $\phi_N$ is the projection of $v$ into the subspace ${\rm span}\left\{c_n(x)e^{\i n\frac{2\pi}{A_1-A_0}}:\,n=-N/2+1,\dots,N/2,\,c_n\in L^2(W)\right\}$, then
\begin{equation}
\phi_N(\alpha,x)=\sum_{j=-N/2+1}^{N/2}\widehat{\phi}_j(x)e^{\i j\alpha\frac{2\pi}{A_1-A_0}}.
\end{equation}
Thus the error between $\phi$ and $\phi_N$ is bounded by
\begin{equation}
\begin{aligned}
\|\phi-\phi_N\|^2_{L^2(\I;L^2(W))}&=\int_\I\int_W|\phi-\phi_N|^2\d x\d \alpha\\
&=\int_\I\int_W\left|\sum_{j\in\Z\setminus[-N/2+1,N/2]}\widehat{\phi}_j(x)e^{\i j\alpha\frac{2\pi}{A_1-A_0}}\right|^2\d x\d \alpha\\
&\leq (A_1-A_0)\sum_{j\in\Z\setminus[-N/2+1,N/2]}\left\|\widehat{\phi}_j\right\|^2_{L^2(W)}.
\end{aligned}
\end{equation}
With the results in Lemma \ref{th:decay_fourier}, 
\begin{equation}
\begin{aligned}
\|\phi-\phi_N\|^2_{L^2(\I;L^2(W))}&\leq (A_1-A_0)\sum_{j\in\Z\setminus[-N/2+1,N/2]}\left|\frac{A_1-A_0}{2\pi j}\right|^{2n+2\gamma}\left\|\phi\right\|^2_{C_0^{n,\gamma}(\I;L^2(W))}\\
&\leq C\left|\frac{A_1-A_0}{2\pi j}\right|^{2n+2\gamma-1}\left\|\phi\right\|^2_{C_0^{n,\gamma}(\I;L^2(W))}.
\end{aligned}
\end{equation}
The case that $m=0$ is proved. The case that $m\in\N$ could be proved in the similar way, thus is omitted here. The proof is finished.

\end{proof}

\subsection{Error estimation}

Let $\I=\Wast$ and $W=\Omega^\Lambda_H$. For the domain $\Omega^\Lambda_H$, let $\mathcal{M}_h$ be a family of regular and quasi-uniform meshes with the mesh width $h\leq h_0$ for some small enough $h_0>0$. We can easily construct the periodic basis functions $\left\{\phi_M^{(\ell)}\right\}_{\ell=1}^M$ that are piecewise linear, globally continuous and vanishes on $\Gamma^\Lambda$. Let $V_h:={\rm span}\left\{\phi_M^{(1)},\dots,\phi_M^{(M)}\right\}$, then $\widetilde{V}_h\subset \widetilde{H}^1_0(\Omega^\Lambda_H)$. A classical error estimate shows that for $v\in H^2(\Omega^\Lambda_H)$ (see \cite{Saute2007}),
\begin{equation}\label{eq:error_fem}
\inf_{v_h\in V_h}\left\|v-v_h\right\|_{H^1(\Omega^\Lambda_H)}\leq Ch\|w\|_{H^2(\Omega^\Lambda_H)}.
\end{equation}

With these basic functions, we can define the finite element space $\widetilde{X}_{N,h}$ by
\begin{equation}
\widetilde{X}_{N,h}=\left\{v_{N,h}(\alpha,x)=e^{-\i\alpha x_1}\sum_{j=1}^N\sum_{\ell=1}^M v_{N,h}^{(j,\ell)}\psi_N^{(j)}(\alpha)\phi_M^{(\ell)}(x):\,v_{N,h}^{(j,\ell)}\in\C\right\}\subset L^2(\Wast;\widetilde{H}^1_\alpha(\Omega^\Lambda_H)).
\end{equation}
Together with the result in Theorem \ref{th:interp_qper} and the classical error estimate \eqref{eq:error_fem}, we could obtain the approximation of function $\phi\in C_0^{n,\gamma}(\I;H^m(S))$, i.e.,
\begin{equation}\label{eq:approx}
\inf_{\phi_{N,h}\in\widetilde{X}_{N,h}}\left\|\phi-\phi_{N,h}\right\|_{L^2(\Wast;\widetilde{H}^1_\alpha(\Omega^\Lambda_H))}\leq C(N^{-n-\gamma+1/2}+h)\|\phi\|_{C_0^{n,\gamma}(\Wast;H^2_\alpha(\Omega^\Lambda_H))}.
\end{equation}

Then we seek for a finite element solution $w_{N,h}\in\widetilde{X}_{N,h}$ to the problem
\begin{equation}\label{eq:var_Bloch_old}
\int_\Wast a_\alpha(w_{N,h}(\alpha,\cdot),z_{N,h}(\alpha,\cdot))\d\alpha+b(w_{N,h},v_{N,h})=\int_\Wast\int_{\Gamma^\alpha_H}F(\alpha,\cdot)\overline{v_{N,h}(\alpha,\cdot)}\d s\d\alpha
\end{equation}
for any $v_{N,h}\in \widetilde{X}_{N,h}$. We do not go to details of the representation of the finite dimensional problem, for this is very complicated but not important for this paper. Following the proof of Theorem 9 in \cite{Lechl2017}, we can prove the high order convergence for the finite element method.

\begin{corollary}\label{th:convergence_high_old}
Suppose $\zeta,\zeta_p\in C^{2,1}(\R)$ and $\left\|\zeta_p-\zeta\right\|_{W^{1,\infty}(\R)}$ is small enough. For any $F\in H_{00}^{n+1}(\Wast;\S;H^{1/2}_\alpha(\Gamma^\Lambda_H))$ with $n\in\N$, there is a unique solution $w_{N,h}$ to the finite dimensional problem \eqref{eq:var_Bloch_old} when $N\geq N_0$ is large enough and $0<h<h_0$ is small enough. Moreover, the solution satisfies
\begin{equation}
\|w_{N,h}-w\|_{L^2(\Wast;H^\ell(\Omega^\Lambda_H))}\leq Ch^{1-\ell}(N^{-r}+h)\|F\|_{C_0^{n,1/2}(\Wast;H^{1/2}_\alpha(\Gamma^\Lambda_H))},\quad\ell=0,1.
\end{equation}
\end{corollary}

\begin{proof}
As $\zeta,\zeta_p\in C^{2,1}(\R)$ and $\zeta_p$ is a small enough perturbation of $\zeta$, for any $F\in H_{00}^{n+1}(\Wast;\S;H^{1/2}_\alpha(\Gamma^\Lambda_H))$, the unique solution $w\in H_0^{n+1}(\Wast;\widetilde{H}^2_\alpha(\Omega^\Lambda_H))$. From Lemma \ref{th:embedding}, $w\in C_0^{n,1/2}(\Wast;\widetilde{H}^2_\alpha(\Omega^\Lambda_H))$. Thus the inequality \eqref{eq:approx} becomes
\begin{equation}
\inf_{w_{N,h}\in \widetilde{X}_{N,h}}\left\|w-w_{N,h}\right\|_{L^2(\Wast;\widetilde{H}^1(\Omega^\Lambda_H))}\leq C(N^{-n}+h)\|w\|_{H^2(\Omega^\Lambda_H)}.
\end{equation}
Then the proof is carried out following the proof of Theorem 9, \cite{Lechl2017}. The proof is finished.
\end{proof}

\subsection{Comments on the numerical method}\label{sec:exp}

It seems that the numerical method introduced in this section is good enough, especially for large $n$'s. However, when applying this classical finite element method to the numerical scheme, the decaying rate of the relative errors is not as fast as expected. As far as we can see, there are two possible reasons for this phenomenon.   

The first reason lies in the approximation of the Dirichlet-to-Neumann map. As the Dirichlet-to-Neumann map is defined as an infinite series, an approximation by finite series is always necessary for numerical schemes. Suppose $M\in\N$ is a large enough positive integer, let the finite approximation be
\begin{equation}
T_\alpha^M\phi:=\i\sum_{j=-M}^M\sqrt{k^2-|\Lambda^*j-\alpha|^2}\widehat{\phi}(\alpha,j)e^{\i(\Lambda^*j-\alpha)x_1}\,\text{ for }\phi=\sum_{j=-M}^M \widehat{\phi}(\alpha,j)e^{\i(\Lambda^*j-\alpha)x_1}.
\end{equation}
From the $\Lambda^*$-periodicity of $\phi(\alpha,\cdot)$, we have
\begin{equation}
\sum_{j\in\Z} \widehat{\phi}(-\Lambda^*/2,j)e^{\i\Lambda^*(j+1/2)x_1}=\sum_{j\in\Z} \widehat{\phi}(\Lambda^*/2,j)e^{\i\Lambda^*(j-1/2)x_1},
\end{equation}
thus
\begin{equation}
\widehat{\phi}(-\Lambda^*/2,j-1)=\widehat{\phi}(\Lambda^*/2,j).
\end{equation}
Then 
\begin{equation}
\begin{aligned}
\left(T^M_\alpha\phi\right)(\Lambda^*/2,x)&=\i\sum_{j=-M}^M\sqrt{k^2-|\Lambda^*(j-1/2)|^2}\,\widehat{\phi}(\Lambda^*/2,j)e^{\i\Lambda^*(j-1/2)x_1}\\
&=\i\sum_{j=-M}^M\sqrt{k^2-|\Lambda^*(j-1/2)|^2}\,\widehat{\phi}(-\Lambda^*/2,j-1)e^{\i\Lambda^*(j-1/2)x_1}\\
&=\i\sum_{j=-M-1}^{M-1}\sqrt{k^2-|\Lambda^*(j+1/2)|^2}\,\widehat{\phi}(-\Lambda,j)e^{\i(\Lambda^*(j+1/2)x_1}.
\end{aligned}
\end{equation}
Thus when $\phi$ is $\Lambda^*-$periodic with respect to $\alpha$, the function $T_\alpha\phi$ fails to be periodic. Thus the periodicity of the solution $w$ with respect to $\alpha$ is not guaranteed, the error estimate in Theorem \ref{th:convergence_high_old} fails.

Another problems may occur in the neighbourhood of the points in $\S$. From \eqref{eq:equivalence}, the solution $w=(\A+\B_p)^{-1}G=(\A^\epsilon+\B_p)^{-1}G$. To guarantee the equivalence, from the requirement of the perturbation theory, the parameter $\epsilon>0$ should be sufficiently small. In this case, the derivative of the cutoff function $\mathcal{X}_\epsilon$ could reach $O(\epsilon^{-1})$, which becomes very large if $\epsilon\rightarrow 0^+$. Thus large oscillations are expect when $\alpha$ lies in the neighborhood of any point in $\S$. Although the solution $w$ may reach a high regularity with respect to $\alpha$, the numerical scheme may result in a bad behavior of relative errors due to the oscillation. A numerical technique is highly demanded for the discretization of functions with large oscillations.

 With these two disadvantages, we apply the method in \cite{Zhang2017e} to modify the original variational problem \eqref{eq:var_Bloch} in the following section.

\section{Modified variational problem}

Due the two possible problems we have explained in the last section, we modify the variational problem \eqref{eq:var_Bloch} in order to deal with the high oscillation problems. We adopt the method introduced in \cite{Zhang2017e} to replace the quasi-periodicity parameter $\alpha$ by a monotonic function. 

For simplicity, we redefine the interval by $\Wast:=(-\underline{k},\Lambda^*-\underline{k}]$, then $\S$ has the following representations 
\begin{itemize}
\item Case 1, $\underline{k}=m\Lambda^*/2$ for some $m=0,1$, $\S=\left\{-\underline{k},\Lambda^*-\underline{k}\right\}$;
\item Case 2, $\underline{k}\neq m\Lambda^*/2$ for any $m=0,1$, $\S=\left\{-\underline{k},\underline{k},\Lambda^*-\underline{k}\right\}$.
\end{itemize}
From the periodicity of $\alpha$, the inverse Bloch transform still has the representation as in Theorem \ref{th:Bloch_property} with the replaced $\Wast$.

Let the $g$ be defined in $\Wast$ and it satisfies the conditions in the following assumption.
\begin{assumption}\label{asp:g}
For any $n\in\N$, suppose $g$ satisfies the following conditions:
\begin{itemize}
\item $g$ is monotonically increasing.
\item $g\in C^{n+1}(\overline{\Wast})$ and $g\in C^\infty(\Wast\setminus\S)$. Thus $g'\in C^{n}(\overline{\Wast})\cap W^{n+1,\infty}(\overline{\Wast})$.
\item $g'(t)>0$ for $t\in\Wast\setminus\S$. Thus the inverse function of $g$ exists in $\Wast\setminus\S$.
\item For any $\alpha_0\in\S$, $g(\alpha_0)=\alpha_0$. Moreover, there is a small enough $\epsilon>0$ such that
\begin{eqnarray}
g(t)=\alpha_0+O\left(|t-\alpha_0|^{n+2}\right)\quad\text{ as }t\rightarrow\alpha_0.
\end{eqnarray}
\end{itemize} 
\end{assumption}
If $g$ satisfies Assumption \ref{asp:g}, $g'\in C^n(\overline{\Wast})$ and $g'(t)=\alpha_0+O\left(|t-\alpha_0|^{n+1}\right)$ as $t\rightarrow\alpha_0$. For example, for $k=1,1.2$, the function $g$ have the graphs in Figure \ref{fig:g}. For the construction of function $g$ with any natural number $n$, we refer to the Appendix.

 \begin{figure}[htb]
\centering
\begin{tabular}{c c}
\includegraphics[width=0.4\textwidth]{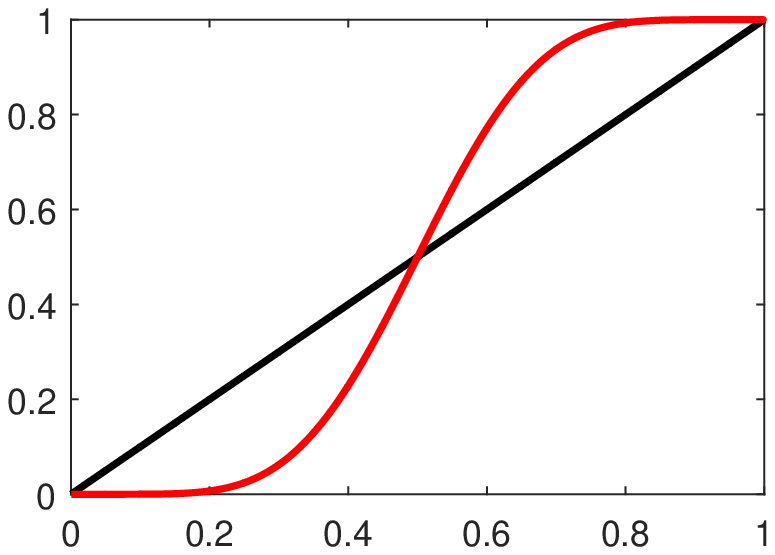} 
& \includegraphics[width=0.4\textwidth]{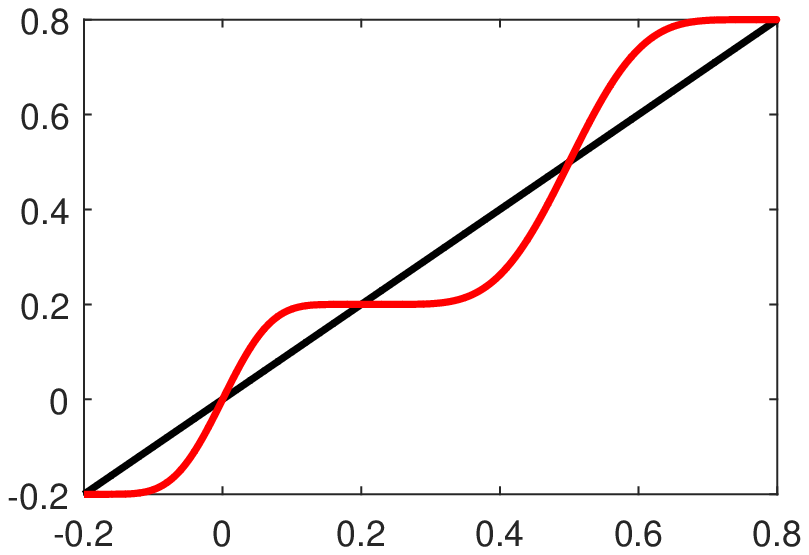}\\[-0cm]
(a) & (b)  
\end{tabular}%
\caption{Examples of function $g$, with $\Lambda=2\pi$, i.e., $\Lambda^*=1$. (a): $k=1$; (b): $k=1.2$.}
\label{fig:g}
\end{figure}
 
 Define the operator $\mathcal{T}$ by
\begin{equation}
\mathcal{T}\phi(t,\cdot)=\phi(g(t),\cdot)g'(t).
\end{equation}
Then we can define the space of the image of $\mathcal{T}$ with domain $L^2(\Wast;H^s_\alpha(\Omega^\Lambda_H))$ by
\begin{equation}
L^2(\Wast;H_{g(t)}^s(\Omega^\Lambda_H)):=\left\{\psi=\mathcal{T}\phi:\,\phi\in L^2(\Wast;H^s_\alpha(\Omega^\Lambda_H))\right\}
\end{equation}
equipped with the norm
\begin{equation}
\|\psi\|^2_{L^2(\Wast;H^s_{g(t)}(\Omega^\Lambda_H))}:=\int_\Wast \|\psi(t,\cdot)\|^2_{H^s_{g(t)}(\Omega^\Lambda_H)}\left[g'(t)\right]^{-1}\d t.
\end{equation}
It is easy to check that
\begin{equation}
\|\mathcal{T}\phi\|_{L^2(\Wast;H^s_{g(t)}(\Omega^\Lambda_H))}=\|\phi\|_{L^2(\Wast;H^s_{\alpha}(\Omega^\Lambda_H))},
\end{equation}
which means that the two function spaces are isometrically isomorphic and $\mathcal{T}$ is an isomporphism with norm equals to $1$. The space $L^2(\Wast;H^s_{g(t)}(\Gamma^\Lambda_H))$ is defined in the same way. The inner product of $L^2(\Wast;{H}^s_{g(t)}(\Omega^\Lambda_H))$ is defined by
\begin{equation}
\left<\phi,\psi\right>_{L^2(\Wast;{H}^s_{g(t)}(\Omega^\Lambda_H))}=\int_\Wast \left<\phi(t,\cdot),\psi(t,\cdot)\right>_{H^s_{g(t)}(\Omega^\Lambda_H)}\left[g'(t)\right]^{-1}\d t.
\end{equation} 
In this section, all the inner product denoted by $\left<\cdot,\cdot\right>$ is  with respect to $L^2(\Wast;H^s_{g(t)}(\Omega^\Lambda_H))$. We can extend the definition of the $L^2-$space to the $H^r-$spaces for $r\in\R$.

                                                                                                                                                                                                                                                                                                                                                                                                                                                                                                                                                                                                                                                                                                                                                                                                                                                                                                                                                                                                                                                                                                                                                                                                                                                                                                                                                                                                                                                                                                                                                                                                                                                                                                                                                                                                                                                                                                                                                                                                                                                                                                                                                                                                                                                                                                                                                                                                                                                                                                                                                                                                                                                                                                                                                                                                                        Let $\alpha:=g(t)$ and $\widetilde{w}(t,\cdot):=\mathcal{T}w\in L^2(\Wast;\S;\widetilde{H}^1_{g(t)}(\Omega^\Lambda_H))$. Then $\widetilde{w}(t,\cdot)$ satisfies the variational problem for any smooth test function $\widetilde{z}(t,\cdot)=\mathcal{T}z$ defined in $\Wast\times\Omega^\Lambda_H$:
\begin{equation}\label{eq:var_Bloch_modi}
\widetilde{a}(\widetilde{w},\widetilde{z})+\widetilde{b}(\widetilde{w},\widetilde{z})=\left<\widetilde{F},\widetilde{z}\right>_{L^2(\Wast;H^{-1/2}_{g(t)}(\Gamma^\Lambda_H))}\quad\text{ where }\widetilde{F}(t,\cdot)=\mathcal{T}F=F(g(t),\cdot)g'(t),
\end{equation}
or equivalently
\begin{equation}\label{eq:var_Bloch_modi_g}
\widetilde{a}(\widetilde{w},\widetilde{z})+\widetilde{b}(\widetilde{w},\widetilde{z})=\left<\widetilde{G},z\right>\quad\text{ where }\widetilde{G}(t,\cdot)=G(g(t),\cdot)g'(t),
\end{equation}
where
\begin{eqnarray}
&&\widetilde{a}(\widetilde{w},\widetilde{z})=\int_\Wast a_{g(t)}(\widetilde{w}(t,\cdot),\widetilde{z}(t,\cdot))\left[g'(t)\right]^{-1}\d t;\\
&&\begin{aligned} \widetilde{b}(\widetilde{w},\widetilde{z})&=\int_{\Omega_H}\left[ (A_p-I_2)\grad\left( \widetilde{\J}_{\Omega_H}^{-1} \widetilde{w}\right)\cdot\grad\overline{\left(\widetilde{\J}_{\Omega_H}^{-1} {\widetilde{z}}\right)}-k^2 (c_p-1)\left(\widetilde{\J}_{\Omega_H}^{-1}\widetilde{w}\right)\cdot\overline{\left(\widetilde{\J}_{\Omega_H}^{-1} {\widetilde{z}}\right)}\right]\d x\\
&=\int_\Wast\int_{\Omega^\Lambda_H}\left[\widetilde{\J}_{\Omega_H}\left[(A_p-I_2)\grad\left(\widetilde{\J}_{\Omega_H}^{-1}\widetilde{w}\right)\right]\cdot\grad\overline{\widetilde{z}}-k^2\widetilde{\J}_{\Omega_H}\left[(c_p-1)\left(\widetilde{\J}^{-1}_{\Omega_H}\widetilde{w}\right)\right]\overline{\widetilde{z}}\right]\d x\d t,
\end{aligned}
\end{eqnarray}
and $\left<\cdot,\cdot\right>$ is the inner product defined in $L^2(\Wast;\widetilde{H}^1_\alpha(\Omega^\Lambda_H))$. The modified  Bloch transforms $\widetilde{\J}_{\Omega_H}^{-1}$ and its inverse $\widetilde{\J}_{\Omega_H}^{-1}$ is defined as
\begin{eqnarray}
&&\left(\widetilde{\J}_{\Omega_H}\phi\right)(t,\cdot)=\left(\J_{\Omega_H}\phi\right)(g(t),\cdot);\\
&&\left(\widetilde{\J}_{\Omega_H}^{-1}\widetilde{\psi}\right)\left(x+\left(\begin{matrix}
\Lambda j\\0
\end{matrix}
\right)\right)=
C_\Lambda\int_\Wast \widetilde{\psi}(t,x)e^{\i g(t)\Lambda j}\d t.
\end{eqnarray}

The property of $\widetilde{\J}_{\Omega_H,N}^{-1}$ could be easily obtained as a corollary of Theorem \ref{th:Bloch_property}, by simply changing variables $\alpha=g(t)$.
\begin{corollary}\label{th:prop_modi_bloch}
The operator $\widetilde{\J}_{\Omega_H}^{-1}$ is bounded from $H^{n'}_{0}(\Wast;H^s_{g(t)}(\Omega^\Lambda_H))$ to $H_{n'}^s(\Omega_H)$  for any $n'\in\N$ such that $n'\leq n$ and $s\in\R$. 
\end{corollary}

The equivalence between \eqref{eq:var_Bloch} and \eqref{eq:var_Bloch_modi_g} is easily proved.
\begin{lemma}
Suppose $\zeta,\,\zeta_p$ are Lipschitz continuous functions. $w\in H_0^{n+1}(\Wast;\widetilde{H}^1_\alpha(\Omega^\Lambda_H))$ is the solution of \eqref{eq:var_Bloch} given any $F\in H_0^{n+1}(\Wast;H^{-1/2}_\alpha(\Gamma^\Lambda_H))$ for some $n\in\N$ if and only if $\mathcal{T}w\in H_0^{n+1}(\Wast;\widetilde{H}^1_{g(t)}(\Omega^\Lambda_H))$ is the solution of \eqref{eq:var_Bloch_modi_g} with  $\widetilde{F}=\mathcal{T}F\in H_0^{n+1}(\Wast;\widetilde{H}^1_{g(t)}(\Omega^\Lambda_H))$.
\end{lemma}

Thus we could obtain the following result as a corollary of Theorem \ref{th:well-posed}.

\begin{theorem}
Suppose $\zeta,\,\zeta_p$ are Lipschitz continuous functions. Then given any $\widetilde{F}\in H_0^{n+1}(\Wast;H^{-1/2}_{g(t)}(\Gamma^\Lambda_H))$, the problem \eqref{eq:var_Bloch_modi_g} is uniquely solvable in $H_0^{n+1}(\Wast;H^{-1/2}_{g(t)}(\Gamma^\Lambda_H))$. Moreover, if $\zeta,\zeta_p\in C^{2,1}(\R)$ and $\widetilde{F}\in H_0^{n+1}(\Wast;H^{1/2}_{g(t)}(\Gamma^\Lambda_H))$, the solution belongs to the space $H_0^{n+1}(\Wast;\widetilde{H}^{2}_{g(t)}(\Gamma^\Lambda_H))$.
\end{theorem}

When $F\in H_{00}^{n+1}(\Wast;\S;H^{-1/2}_\alpha(\Gamma^\Lambda_H))$ and the perturbation $\zeta_p-\zeta$ is small enough, the solution of \eqref{eq:var_Bloch_modi_g} also belongs to a H\"{o}lder space with respect to $\alpha$.

\begin{theorem}\label{th:well_posed_continuous}
Suppose $\zeta,\zeta_p$ satisfy the conditions of Theorem \ref{th:well-posed_high_order}. Given $\widetilde{F}=\mathcal{T}F$ for any $F\in H_{00}^{n+1}(\Wast;\S;H^{-1/2}_\alpha(\Gamma^\Lambda_H))$ where $n\in\N$, then the unique solution $\widetilde{w}\in C_0^{n,1/2}(\Wast;\widetilde{H}^1_{g(t)}(\Omega^\Lambda_H))$.
\end{theorem}

\begin{proof}
Let $w\in H_0^{n+1}(\Wast;\widetilde{H}^1_\alpha(\Omega^\Lambda_H))$ be the unique solution of \eqref{eq:var_Bloch} for some $F\in H_{00}^r(\Wast;\S;H^{-1/2}_\alpha(\Gamma^\Lambda_H))$. Then the unique solution $w\in H_0^{n+1}(\Wast;\widetilde{H}^1_\alpha(\Omega^\Lambda_H))$. From Sobolev embedding proved in Lemma \ref{th:embedding}, $H_0^{n+1}(\Wast;\widetilde{H}^1_\alpha(\Omega^\Lambda_H))\subset C_0^{n,1/2}(\Wast;\widetilde{H}^1_\alpha(\Omega^\Lambda_H))$, thus $w\in C_0^{n,1/2}(\Wast;\widetilde{H}^1_\alpha(\Omega^\Lambda_H))$. As $g\in C^{n+1}_0(\overline{\Wast})$, the composition $w(g(t),\cdot)\in C_0^{n,1/2}(\Wast;\widetilde{H}^1_{g(t)}(\Omega^\Lambda_H))$. 
  Moreover, there is a constant $C>0$ that does not depend on $F$ such that
\begin{equation}
\|\widetilde{w}\|_{ C_0^{n,1/2}(\Wast;\widetilde{H}^1_\alpha(\Omega^\Lambda_H))}\leq C\|F\|_{H_{00}^{n+1}(\Wast;\S;H^{-1/2}_\alpha(\Gamma^\Lambda_H))}.
\end{equation} 
The proof is finished.

\end{proof}

We conclude the above results as follows.

\begin{theorem}\label{th:final_well_posed}
Suppose the following conditions are satisfied:
\begin{itemize}
\item $\zeta,\,\zeta_p\in C^{2,1}(\R)$ and $\|\zeta_p-\zeta\|_{W^{1,\infty}(\R)}$ is small enough.
\item $F\in H_{00}^{n+1}(\Wast;\S;H^{1/2}_\alpha(\Gamma^\Lambda_H))$ and $\widetilde{F}=\widetilde{T}F$.
\item $g$ satisfies Assumption \ref{asp:g} for $n\in\N$.
\end{itemize}
Then there is a unique solution $\widetilde{w}\in H_0^{n+1}(\Wast;\widetilde{H}^2_{g(t)}(\Omega^\Lambda_H))$. Moreover, the unique solution $\widetilde{w}\in C_0^{n,1/2}(\Wast;\widetilde{H}^2_\alpha(\Omega^\Lambda_H))$. 
\end{theorem}

We could return to the two problems we have mentioned in Section \ref{sec:exp}. The first is the truncated Dirichlet-to-Neumann map. Define $T^M_{g(t)}$ by 
\begin{equation}
T^M_{g(t)}\widetilde{\phi}:=\i\sum_{j=-M}^M\sqrt{k^2-|\Lambda^*j-g(t)|^2}\widetilde{\phi}(g(t),j)e^{\i(\Lambda^*j-g(t))x_1}g'(t)
\end{equation}
for $\widetilde{\phi}=\phi(g(t),\cdot)g'(t)$. As $g'(-\underline{k},\Lambda^*-\underline{k})=0$, no matter which positive integer $M$ is chosen, for any $\widetilde{\phi}$, 
\begin{equation}
T^M_{g(t)}\widetilde{\phi}(t,\cdot)\Big|_{t=-\underline{k}}=T^M_{g(t)}\widetilde{\phi}(t,\cdot)\Big|_{t=\Lambda^*-\underline{k}}.
\end{equation}
Thus the periodicity of $T^M_{g(t)}\phi$ is guaranteed. 

The second problem is the high oscillation of the derivative of $w(\alpha,\cdot)$ for $\alpha$ in the neighbourhood of $\S$. 
\begin{equation}
\frac{\partial\widetilde{w}(t,x)}{\partial t}=\frac{\partial w(g(t),\cdot)}{\partial_\alpha}\left[g'(t)\right]^2+w(g(t),\cdot)g''(t).
\end{equation}
From $g(t)=\alpha_0+O(|t-\alpha_0|^{n+2})$, $g'(t)=O(|t-\alpha_0|^{n+1})$ at the points of $\S$. As $g'(t)$ tends to $0$ as $t\rightarrow \alpha_0$ for any $\alpha_0\in\S$, $\frac{\partial\widetilde{w}(t,\cdot)}{\partial t}$ also tends to $0$. This implies that the value of the derivative will not be very large in the neighbourhood of $\alpha_0$, thus it is reasonable that the oscillation will not be too large. So it is possible that the numerical method based on the modified problems will possess a good performance.

\section{The finite section method}

For numerical simulations in unbounded domains, a cutoff technique is always highly demanded. For rough surface scattering problem, the finite section method (see \cite{Chand2010}) was discussed. The convergence of the finite section method depends on the decay index $r$ for the incident field. In this section, we apply the finite section method to the original variational problem \eqref{eq:var_Bloch} and develop higher convergence rate for the special cases, then with the help of the change of variables, to extend the finite section method to \eqref{eq:var_Bloch_modi_g}.

For any $L\in\N$ large enough, define the finite domain by
\begin{equation}
\Omega^{L,\Lambda}_H:=\bigcup_{j=-L+1}^L\Big[\Omega^\Lambda_H+\left(\begin{smallmatrix}
\Lambda j\\0
\end{smallmatrix}
\right)\Big].
\end{equation}
Define the cutoffed functions
\begin{equation}
A_p^L(x):=\begin{cases}
A_p(x),\quad\text{ if }x\in\Omega^{L,\Lambda}_H; \\
I_2,\quad\text{ otherwise;}
\end{cases},\quad
c_p^L(x):=\begin{cases}
c_p(x),\quad\text{ if }x\in\Omega^{L,\Lambda}_H; \\
1,\quad\text{ otherwise.}
\end{cases}
\end{equation}
Thus both of the functions $A_p^L-I_2$ and $c_p^L-1$ are  compactly supported in the bounded domain $\Omega^{L,\Lambda}_H$.  Replace $A_p$ and $c_p$ by $A_p^L$ and $c_p^L$ in \eqref{eq:var_Bloch}, the only term to be changed is $b(w,z)$, denoted by $b^L(w,z)$. In this case, this term becomes
\begin{equation}
\begin{aligned}
b^L(w,z)&=\int_{\Omega_H}\Big[ (A_p^L-I_2)\grad( \J_{\Omega_H}^{-1} w)\cdot\grad\overline{\left(\J_{\Omega_H}^{-1} \overline{z}\right)}-k^2 (c_p^L-1)(\J_{\Omega_H}^{-1}w)\cdot\overline{\left(\J_{\Omega_H}^{-1} \overline{z}\right)}\Big]\d x\\
&=\int_{\Omega_H^{L,\Lambda}}\Big[ (A_p-I_2)\grad( \J_{\Omega_H}^{-1} w)\cdot\grad\overline{\left(\J_{\Omega_H}^{-1} \overline{z}\right)}-k^2 (c_p-1)(\J_{\Omega_H}^{-1}w)\cdot\overline{\left(\J_{\Omega_H}^{-1} \overline{z}\right)}\Big]\d x.
\end{aligned}
\end{equation}
Apply the property of the Bloch transform, i.e., it commutes with the gradient operator and $\J_{\Omega_H}^{-1}=\J_{\Omega_H}^*$, 
\begin{equation}
\begin{aligned}
b^L(w,z)=&\int_\Wast\int_{\Omega^\Lambda_H}\J_{\Omega_H}\left((A_p^L-I_2)\grad\left(\J_{\Omega_H}^{-1}w\right)\right)\cdot\overline{\grad z}\d x\d\alpha\\
&-k^2 \int_\Wast\int_{\Omega^\Lambda_H}\J_{\Omega_H}\left((c_p^L-1)\left(\J_{\Omega_H}^{-1}w\right)\right)\cdot\overline{z}\d x\d\alpha.
\end{aligned}
\end{equation}
From similar analysis of $b(w,z)$ in \cite{Zhang2018e}, $b^L(w,z)$ is a bounded linear operator in  $H_0^n(\Wast;\widetilde{H}^1_\alpha(\Omega^\Lambda_H))\times H_0^{-n}(\Wast;\widetilde{H}^1_\alpha(\Omega^\Lambda_H))$. The estimate between $b$ and $b^L$ is described in the following lemma. 

\begin{lemma}\label{th:error_cut}
When $L$ is large enough, for any $w\in H_0^n(\Wast;\widetilde{H}^1_\alpha(\Omega^\Lambda_H))$ and $z\in L^2(\Wast;\widetilde{H}^1_\alpha(\Omega^\Lambda_H))$, there is a constant $C>0$ that does not depend on $L$ such that
\begin{equation}
\left|b^L(w,z)-b(w,z)\right|\leq C|L\Lambda|^{-n}\|w\|_{H_0^n(\Wast;\widetilde{H}^1_\alpha(\Omega^\Lambda_H))} \|z\|_{L^2(\Wast;\widetilde{H}^1_\alpha(\Omega^\Lambda_H))}.
\end{equation}
\end{lemma}

\begin{proof}
First, let $w,z\in C^\infty(\Wast\times\Omega^\Lambda_H)$ that are periodic in $\alpha$ and $\alpha$-quasi-periodic in $x$ and vanishes on $\Wast\times\Gamma^\Lambda$. Then $w\in H_0^n(\Wast;L^2(\Omega^\Lambda_H))$, thus $\J_{\Omega_H}^{-1}w\in H_r^0(\Omega_H)$, i.e.,
\begin{equation}
\left\|\J_{\Omega_H}^{-1}w\right\|^2_{H_r^0(\Omega_H)}=\int_{\Omega_H}(1+|x|^2)^n\left|\J_{\Omega_H}^{-1}w\right|^2\d x<\infty.
\end{equation}
The difference between $b^L\cdot,\cdot)$ and $b(\cdot,\cdot)$ has the following representation
\begin{equation}
b^L(w,z)-b(w,z)=\int_{\Omega_H\setminus\Omega^{L,\Lambda}_H}\Big[\J_{\Omega_H}\left ((A_p-I_2)\grad( \J_{\Omega_H}^{-1} w)\right)\cdot\grad\overline{z}-k^2 \J_{\Omega_H}\left ((c_p-1)(\J_{\Omega_H}^{-1}w)\right)\cdot\overline{z}\Big]\d x.
\end{equation}
Take the second term as an example.
\begin{equation}
\begin{aligned}
\left\| (c_p-1)(\J_{\Omega_H}^{-1}w)\right\|^2_{L^2(\Omega_H\setminus\Omega^{L,\Lambda}_H)}\leq& C\left\|\J_{\Omega_H}^{-1}w\right\|^2_{L^2(\Omega_H\setminus\Omega^{L,\Lambda}_H)}=C\int_{\Omega_H\setminus\Omega^{L,\Lambda}_H}\left|\J_{\Omega_H}^{-1}w\right|^2\d x\\
\leq & C\sup_{x\in \Omega_H\setminus\Omega^{L,\Lambda}_H}(1+|x|^2)^{-n} \int_{\Omega_H\setminus\Omega^{L,\Lambda}_H}(1+|x|^2)^n\left|\J_{\Omega_H}^{-1}w\right|^2\d x\\
\leq & C|L\Lambda|^{-2n}\left\| \J_{\Omega_H}^{-1}w\right\|^2_{H_r^0(\Omega_H)}\leq C|N\Lambda|^{-2n}\|w\|^2_{H_0^n(\Wast;L^2(\Omega^\Lambda_H))}.
\end{aligned}
\end{equation}
Thus
\begin{equation}
\left|\int_{\Omega_H\setminus\Omega^{L,\Lambda}_H}\J_{\Omega_H}\left ((c_p-1)(\J_{\Omega_H}^{-1}w)\right)\cdot\overline{z}\d x\right|\leq C|L\Lambda|^{-n}\|w\|_{H_0^n(\Wast;L^2(\Omega^\Lambda_H))}\|z\|_{L^2(\Wast;L^2(\Omega^\Lambda_H))}.
\end{equation}
From similar argument, we can also estimate the first term, i.e.,
\begin{equation}
\left|\J_{\Omega_H}\left ((A_p-I_2)\grad( \J_{\Omega_H}^{-1} w)\right)\cdot\grad\overline{z}\right|\leq C|L\Lambda|^{-n}\|w\|_{H_0^n(\Wast;L^2(\Omega^\Lambda_H))}\|z\|_{L^2(\Wast;L^2(\Omega^\Lambda_H))}.
\end{equation}
Thus 
\begin{equation}
\left|b^L(w,z)-b(w,z)\right|\leq C|L\Lambda|^{-n}\|w\|_{H_0^n(\Wast;L^2(\Omega^\Lambda_H))}\|z\|_{L^2(\Wast;L^2(\Omega^\Lambda_H))}.
\end{equation}
From the denseness of $C^\infty(\Wast\times\Omega^\Lambda_H)$, the result could be easily extended to $w\in H_0^n(\Wast;\widetilde{H}^1_\alpha(\Omega^\Lambda_H))$ and $z\in L^2(\Wast;\widetilde{H}^1_\alpha(\Omega^\Lambda_H))$. The proof is finished.
\end{proof}

 We introduce the new variational problem, i.e., to find a $w^L\in H_0^n(\Wast;\widetilde{H}^1_\alpha(\Omega^\Lambda_H))$ such that
\begin{equation}\label{eq:var_Bloch_cutoff}
\int_\Wast a_\alpha(w^L(\alpha,\cdot),z(\alpha,\cdot))\d\alpha+b^L(w^L,z)=\int_\Wast\int_{\Gamma^\Lambda_H}F(\alpha,x)\overline{z(\alpha,x)}\d s\d\alpha.
\end{equation} 

\begin{theorem}\label{th:cutoff_problem_error}
Suppose $\zeta,\zeta_p$ satisfy conditions in Theorem \ref{th:well-posed_high_order}. When $L\in\N$ is large enough, given an $F\in H_{00}^n(\Wast;\S;H^{-1/2}_\alpha(\Gamma^\Lambda_H))$, there is a unique $w^L\in H_0^n(\Wast;\widetilde{H}^1_\alpha(\Omega^\Lambda_H))$ satisfies \eqref{eq:var_Bloch_cutoff}.
Moreover, there is a $C>0$ that does not depend on $L$ such that.
\begin{equation}
\|w^L-w\|_{L^2(\Wast;\widetilde{H}^1_\alpha(\Omega^\Lambda_H))}\leq C|L\Lambda|^{-n}\|w\|_{H_0^n(\Wast;\widetilde{H}^1_\alpha(\Omega^\Lambda_H))}
\end{equation}
\end{theorem}

\begin{proof}
From the proof of Theorem 24, \cite{Zhang2018e}, when $\zeta_p-\zeta$ is sufficiently small, i.e., when $\|A_p-I_2\|_\infty,\,\|c_p-1\|_\infty$ are small enough, the problem is uniquely solvable in $H_0^n(\Wast;\widetilde{H}^1_\alpha(\Omega^\Lambda_H))$. From the definition of the cutoffed functions $A_p^L$ and $c_p^L$, when \eqref{eq:var_Bloch} is uniquely solvable, the problem \eqref{eq:var_Bloch_cutoff} is also uniquely solvable.  Moreover, there is a constant $C>0$ that does not depend on $L$ such that
\begin{equation}
\|w^L\|_{H_0^n(\Wast;\widetilde{H}^1_\alpha(\Omega^\Lambda_H))}\leq C\|F\|_{H_{00}^n(\Wast;\S;H^{-1/2}_\alpha(\Gamma^\Lambda_H))}.
\end{equation}
 Subtracting \eqref{eq:var_Bloch} from \eqref{eq:var_Bloch_cutoff} and let $w_D:=w^L-w$, then $w_D$ satisfies
\begin{equation}
\int_\Wast a_\alpha(w_D(\alpha,\cdot),z(\alpha,\cdot))\d\alpha+b^L(w_D,z)=b(w,z)-b^L(w,z).
\end{equation}
From Lemma \ref{th:error_cut}, the right hand side is bounded by
\begin{equation}
\left|b(w,z)-b^L(w,z)\right|\leq C|L\Lambda|^{-n}\|w\|_{H_0^n(\Wast;\widetilde{H}^1_\alpha(\Omega^\Lambda_H))}\|z\|_{L^2(\Wast;\widetilde{H}^1_\alpha(\Omega^\Lambda_H))}.
\end{equation}
Thus $w_D$ is the unique solution satisfies
\begin{equation}
\|w_D\|_{L^2(\Wast;\widetilde{H}^1_\alpha(\Omega^\Lambda_H))}\leq C|L\Lambda|^{-n}\|w\|_{H_0^n(\Wast;\widetilde{H}^1_\alpha(\Omega^\Lambda_H))}.
\end{equation}
The proof is finished.
\end{proof}

Theorem \ref{th:cutoff_problem_error} shows that the solution to the new problem \eqref{eq:var_Bloch_cutoff} is a good approximation of the original one, especially when $r$ is large. Now we are prepared to apply the truncation to the modified problem \eqref{eq:var_Bloch_modi_g}. We define the truncated sesquilinear form $\widetilde{b}^L(\cdot,\cdot)$ by
\begin{equation}
\widetilde{b}^L(\widetilde{w},\widetilde{z})=\int_{\Omega_H^{L,\Lambda}}\left[ (A_p-I_2)\grad\left( \widetilde{\J}_{\Omega_H}^{-1} \widetilde{w}\right)\cdot\grad\overline{\left(\widetilde{\J}_{\Omega_H}^{-1} {\widetilde{z}}\right)}-k^2 (c_p-1)\left(\widetilde{\J}_{\Omega_H}^{-1}\widetilde{w}\right)\cdot\overline{\left(\widetilde{\J}_{\Omega_H}^{-1} {\widetilde{z}}\right)}\right]\d x.
\end{equation}
Consider the variational problem 
\begin{equation}\label{eq:var_Bloch_modi_finite}
\widetilde{a}(\widetilde{w}^L,\widetilde{z})+\widetilde{b}^L(\widetilde{w}^L,\widetilde{z})=\left<\widetilde{F},\widetilde{z}\right>,
\end{equation}
for any $\widetilde{z}\in L^2(\Wast;\widetilde{H}^1_{g(t)}(\Omega^\Lambda_H))$. Thus this is equivalent to \eqref{eq:var_Bloch_cutoff} by simply let $\alpha=g(t)$. Thus we have the following well-posedness property of this new variational problem.

\begin{theorem}\label{th:cutoff_err}
Suppose all the  conditions of Theorem \ref{th:final_well_posed} are satisfied. For sufficient large $L$, the variational problem \eqref{eq:var_Bloch_modi_finite} is uniquely solvable. Moreover, the unique solution $\widetilde{w}^L\in C_0^{n,1/2}(\Wast;\widetilde{H}^2_{g(t)}(\Omega^\Lambda_H))$ satisfies
\begin{equation}
\|\widetilde{w}^L-\widetilde{w}\|_{L^2(\Wast;\widetilde{H}^1_{g(t)}(\Omega^\Lambda_H))}\leq C|L\Lambda|^{-n}\|\widetilde{w}\|_{H_0^n(\Wast;\widetilde{H}^1_{g(t)}(\Omega^\Lambda_H))}.
\end{equation}
\end{theorem}

Now we are prepared to apply the finite element method to the numerical simulation for the scattering problems. We apply the discretization technique to the variational problem \eqref{eq:var_Bloch_modi_finite}. In the next section, we will introduce the Galerkin method for the numerical solution, and investigate the error estimation of the method.

\section{The finite element method}

In this section, we describe the numerical method to solve the variational problem \eqref{eq:var_Bloch_modi_finite}. The Galerkin method based on the Bloch transform has been applied to numerical solutions of the scattering problems with (locally perturbed) periodic background in \cite{Lechl2016a,Lechl2016b,Lechl2017}. It is also extended to rough surface scattering problems in \cite{Zhang2018c}. In this section, we apply the method to the solution of \eqref{eq:var_Bloch_modi} again, and a high order convergence rate could be obtained for special cases.

We still denote by $\left\{\phi_M^{(\ell)}\right\}^M_{\ell=1}$ the periodic, piecewise linear and globally continuous basis in $\Omega^\Lambda_H$ and $\left\{\psi_N^{(j)}\right\}^N_{j=1}$ the basis in $\Wast$. Let $V_h:=\left\{\phi_M^{(\ell)}\right\}^M_{\ell=1}$, then $V_h\subset \widetilde{H}^1_0(\Omega^\Lambda_H)$. Note that the variational problem \eqref{eq:var_Bloch_modi_finite} is equivalent to find $\widetilde{w}^L\in L^2(\Wast;\widetilde{H}^1_{g(t)}(\Omega^\Lambda_H))$ such that
\begin{equation}\label{eq:var_Bloch_finite_equiv}
a_0(\widetilde{w}^L,\widetilde{z})+b^L_0(\widetilde{w}^L,\widetilde{z})=\left<\widetilde{F},\widetilde{z}\right>
\end{equation}
for any $\widetilde{z}(t,\cdot)=z(g(t),\cdot)$ with $z\in L^2(\Wast;\widetilde{H}^1_\alpha(\Omega^\Lambda_H))$, where
\begin{eqnarray}
&& a_0(\widetilde{w}^L,\widetilde{z})=\int_\Wast a_{g(t)}(\widetilde{w}(t,\cdot),\widetilde{z}(t,\cdot))\d t;\\
&& b_0^L(\widetilde{w}^L,\widetilde{z})=\widetilde{b}^L(\widetilde{w}^L,\widetilde{z}(t,\cdot)g'(t)).
\end{eqnarray} 

We look for the solution of \eqref{eq:var_Bloch_finite_equiv} in the following finite element space
\begin{equation}
X_{N,h}=\left\{v_{N,h}(t,x)=e^{-\i g(t) x_1}\sum_{j=1}^N\sum_{\ell=1}^M v_{N,h}^{(j,\ell)}\psi_N^{(j)}(t)\phi_M^{(\ell)}(x):\,v_{N,h}^{(j,\ell)}\in\C\right\},
\end{equation}
then it is a subspace of $\left\{\phi(g(t),\cdot):\,\phi\in L^2(\Wast;\widetilde{H}^1_\alpha(\Omega^\Lambda_H))\right\}$.

 Let $w_{N,h}^L$ be the approximation of $w^L(g(t),\cdot)$ in $X_{N,h}$, then it is easy to check that
\begin{equation}
\begin{aligned}
\|w_{N,h}^L(t,\cdot)g'(t)-\widetilde{w}^L\|_{L^2(\Wast;\widetilde{H}^1_{g(t)}(\Omega^\Lambda_H))}&=\int_\Wast \|w_{N,h}^L(t,\cdot)-{w}^L(g(t),\cdot)\|_{L^2(\Wast;\widetilde{H}^1_{g(t)}(\Omega^\Lambda_H))}g'(t)\d t\\
&\leq C\int_\Wast \|w_{N,h}^L(t,\cdot)-{w}^L(g(t),\cdot)\|_{L^2(\Wast;\widetilde{H}^1_{g(t)}(\Omega^\Lambda_H))}\d t\\
&\leq C\|w_{N,h}^L(t,\cdot)-w^L(g(t),\cdot)\|_{L^2(\Wast;H^1(\Omega^\Lambda_H))}.
\end{aligned}
\end{equation}
When $w\in H_0^{n+1}(\Wast;\widetilde{H}^1_\alpha(\Omega^\Lambda_H))$ and $g$ satisfies Assumption \ref{asp:g}, from the proof of Theorem \ref{th:well_posed_continuous}, $w(g(t),\cdot)\in C_0^{n,1/2}(\Wast;\widetilde{H}^1_{g(t)}(\Omega^\Lambda_H))$. From theorem \ref{th:interp_qper}, together with \eqref{eq:approx}, when $w\in H_0^{n+1}(\Wast;\widetilde{H}^2_\alpha(\Omega^\Lambda_H))$,  the interpolation $w_{N,h}$ satisfies
\begin{equation}
\|w_{N,h}^L-w^L(g(t),\cdot)\|_{L^2(\Wast;H^1(\Omega^\Lambda_H))}\leq C\left(N^{-n}+h\right)\|\widetilde{w}^L(g(t),\cdot)\|_{C_0^{n,1/2}(\Wast;H^2(\Omega^\Lambda_H))}.
\end{equation}
This implies that
\begin{equation}
\begin{aligned}
\|w_{N,h}^L(t,\cdot)g'(t)-\widetilde{w}^L\|_{L^2(\Wast;\widetilde{H}^1_{g(t)}(\Omega^\Lambda_H))}\leq &C\|w_{N,h}^L(t,\cdot)g'(t)-\widetilde{w}^L\|_{L^2(\Wast;\widetilde{H}^1(\Omega^\Lambda_H)}\\
\leq &C\left(N^{-n}+h\right)\|\widetilde{w}^L(g(t),\cdot)\|_{C_0^{n,1/2}(\Wast;H^2(\Omega^\Lambda_H))}.
\end{aligned}
\end{equation}

Let
\begin{equation}
w^L_{N,h}(t,x):=e^{-\i g(\alpha_N^{(j)}) x_1}\sum_{j=1}^N\sum_{\ell=1}^M w_{N,h}^{(j,\ell)}\psi_N^{(j)}(t)\phi_M^{(\ell)}(x),
\end{equation}
then it satisfies the finite discrete variational problem
\begin{equation}\label{eq:var_Bloch_disc}
\begin{aligned}
a_{g(t)}(w^L_{N,h},z_{N,h})
&+\int_{\Omega^\Lambda_H}\left[\widetilde{\J}_{\Omega_H}\left[(A^L_p-I_2)\grad\left(\widetilde{\J}_{\Omega_H}^{-1}w^L_{N,h}\right)\right]\cdot\grad\overline{z_{N,h}}\right]\d x\\
&-k^2\int_{\Omega^\Lambda_H}\left[\widetilde{\J}_{\Omega_H}\left[(c_p^L-1)\left(\widetilde{\J}_{\Omega_H}^{-1}w^L_{N,h}\right)\right]\overline{z_{N,h}}\right]\d x=\left<F(g(t),\cdot),z_{N,h}\right>,
\end{aligned}
\end{equation}
for any $z_{N,h}\in X_{N,h}$. Following the analysis of the finite element method in \cite{Lechl2017}, we can finally obtain the convergence of the numerical method.

\begin{theorem}\label{th:converge_final}
Suppose $u^i\in H_{n+1}^2(\Omega^p_H)$ for $n\in\N$. Moreover, all the conditions in Theorem \ref{th:final_well_posed} are satisfied. Then there is a unique solution $w_{N,h}^L$ for the variational problem \eqref{eq:var_Bloch_disc}. The error between the numerical solution and the exact one $\widetilde{w}$ is bounded by
\begin{equation}
\|w^L_{N,h}g'(t)-\widetilde{w}\|_{L^2(\Wast;\widetilde{H}^\ell_{g(t)}(\Omega^\Lambda_H))}\leq C\left(h^{1-\ell}N^{-n}+h^{2-\ell}+|L\Lambda|^{-n-1}\right)\|\widetilde{w}\|_{C^{n,1/2}_0(\Wast;\widetilde{H}^2_{g(t)}(\Omega^\Lambda_H))}
\end{equation}
for $\ell=0,1$.
\end{theorem}

\begin{proof}
From the arguments above, as $w_{N,h}^L(t,\cdot)g'(t)$ is the solution \eqref{eq:var_Bloch_modi_finite}, $\|w_{N,h}^L(t,\cdot)g'(t)-\widetilde{w}^L\|_{L^2(\Wast;\widetilde{H}^1_{g(t)}(\Omega^\Lambda_H))}\leq C\left(N^{-n}+h\right)\|\widetilde{w}^L(g(t),\cdot)\|_{C_0^{n,1/2}(\Wast;H^2(\Omega^\Lambda_H))}$.
From Theorem \ref{th:cutoff_err},the relative error of $\widetilde{w}^L$ is bounded by $\|\widetilde{w}^L-\widetilde{w}\|_{L^2(\Wast;\widetilde{H}^1_{g(t)}(\Omega^\Lambda_H))}\leq C|L\Lambda|^{-n-1}\|\widetilde{w}\|_{H_0^{n+1}(\Wast;\widetilde{H}^2_{g(t)}(\Omega^\Lambda_H))}$. The proof is finished by the triangular inequality.

\end{proof}

The approximated total field, denoted by $u^L_{N,h}$ is defined as
\begin{equation}
u^L_{N,h}\left(x+\left(\begin{matrix}
\Lambda j\\0
\end{matrix}\right)\right)=\left(\widetilde{\J_{\Omega_H}^{-1}}w^L_{N,h}\right)\left(x+\left(\begin{matrix}
\Lambda j\\0
\end{matrix}\right)\right).
\end{equation}
Then the relative error of $u^L_{N,h}$ is concluded in the following corollary.

\begin{corollary}\label{th:converge_total}
The numerical solution of the total field is bounded by
\begin{equation}
\|u^L_{N,h}-u_T\|_{H^\ell(\Omega_H)}\leq C\left(h^{1-\ell}N^{-n}+h^{2-\ell}+|L\Lambda|^{-n-1}\right)\|\widetilde{w}\|_{C^{n,1/2}_0(\Wast;\widetilde{H}^2_{g(t)}(\Omega^\Lambda_H))}.
\end{equation}
\end{corollary}

\begin{proof}
From Theorem \ref{th:converge_final}, the error of $w^L_{N,h}g'(t)$ is bounded by
\begin{equation}
\|w^L_{N,h}g'(t)-\widetilde{w}\|_{L^2(\Wast;\widetilde{H}^\ell_{g(t)}(\Omega^\Lambda_H))}\leq C\left(hN^{-n}+h^2+|L\Lambda|^{-n-1}\right)\|\widetilde{w}\|_{C^{n,1/2}_0(\Wast;\widetilde{H}^2_{g(t)}(\Omega^\Lambda_H))}.
\end{equation}
From Corollary \ref{th:prop_modi_bloch}, $\widetilde{\J}_{\Omega_H}^{-1}$ is bounded from $L^2(\Wast;H^\ell_{g(t)}(\Omega^\Lambda_H))$ to $H^\ell(\Omega_H)$. Then
\begin{equation}
\begin{aligned}
\|u^L_{N,h}-u_T\|_{H^\ell(\Omega_H)}&=\|\widetilde{\J}_{\Omega_H}^{-1}\left(w^L_{N,h}g'(t)-\widetilde{w}\right)\|_{H^\ell(\Omega_H)}\\
&\leq C\|w^L_{N,h}g'(t)-\widetilde{w}\|_{L^2(\Wast;H^\ell_{g(t)}(\Omega_H^\Lambda))}.
\end{aligned}
\end{equation}
The proof is finished from the result in Theorem \ref{th:converge_final}.
\end{proof}

\section{Numerical results}

In this paper, we present several numerical experiments to show that the convergence rate of algorithm could reach the results in Corollary \ref{th:converge_total}. We fix $n=5$, and choose $g$ as in the Appendix. The period $\Lambda$ is fixed as $2\pi$ thus $\Lambda^*=1$. The parameters are chosen as follows
\begin{equation}
H=3;\,H_0=2.9;\, 
\end{equation}
In the numerical examples, we fix $L=N/2$. The incident field is always the downward propagating Herglotz wave function
\begin{equation}
u^i(x)=\int_{-\pi/2}^{\pi/2}e^{\i k (x_1\sin t-x_2\cos t)}h(t)\d t
\end{equation}
where
\begin{equation}
h(t)=\widetilde{h}(t)/\|h\|_\infty\text{ for }h(t)=\begin{cases}
(t-0.5)^6(t-1.3)^6\quad\text{ when }0.5<t<1.3;\\
(t+0.5)^6(t+1.3)^6\quad\text{ when }-1.3<t<-0.5;\\
0\quad\text{ otherwise}.
\end{cases}
\end{equation}
From \cite{Lechl2015e}, the Bloch transform of the incident field has the form of
\begin{equation}
\left(\J_{\Omega}u^i\right)(\alpha,x)=\sum_{j\in\Z:\,|j-\alpha|<k}e^{\i (j-\alpha)x_1-\i\sqrt{k^2-|j-\alpha|^2}\,x_2}\frac{h\left[\arcsin((j-\alpha)/k)\right]}{\sqrt{k^2-|j-\alpha|^2}}.
\end{equation}
For all the examples in this section, we choose three different $k$'s, i.e., $k=1,\sqrt{2},1.5$. It is easy to check that for any of these $k$'s, the Bloch transform of $u^i$ vanishes in the neighborhood of $\S$, which is the set depends on $k$.\\

\noindent {\bf Example 1. }
The first group contains two examples. The periodic surface is a straight line $\R\times\{1\}$ and the perturbation is also a straight line $\R\times\{1.1\}$, i.e., $\zeta(t)=1$ and $\zeta_p(t)=1.1$. In this case, the total field is easily obtained, i.e.,
\begin{equation}
u(x)=\int_{-\pi/2}^{\pi/2}\left[e^{\i k (x_1\sin t-x_2\cos t)}-e^{\i k (x_1\sin t+(x_2-2.2)\cos t)}\right]h(t)\d t.
\end{equation}
We fix $h=0.025$ for the numerical experiments, thus the error caused by this parameter is about $O(10^{-4})$. For $N=4,8,16,32,64$, we compute the total field $u^L_{N,h}$ and the relative error
\begin{equation}
\frac{\left\|u^L_{N,h}-u\right\|_{L^2(\Gamma^\Lambda_H)}}{\|u\|_{L^2(\Gamma^\Lambda_H)}}.
\end{equation}
The errors are listed in Table \ref{eg1}. From the table we can find that the relative error decays very fast at first. But when it reaches $O(10^{-4})$, the relative becomes stable. This implies that the errors brought by $N$ and $L$ are minor compared to that brought by $h$. 

\begin{table}[htb]
\centering
\begin{tabular}
{|p{1.8cm}<{\centering}||p{2cm}<{\centering}|p{2cm}<{\centering}
 |}
\hline
  & $k=1$ & $k=\sqrt{2}$\\
\hline
\hline
$N=8$&$3.77$E$-01$&$2.38$E$-01$\\
\hline
$N=16$&$9.77$E$-03$&$3.20$E$-04$\\
\hline
$N=32$&$9.92$E$-05$&$1.25$E$-04$\\
\hline
$N=64$&$9.54$E$-05$&$1.25$E$-04$\\
\hline
\end{tabular}
\caption{Relative $L^2$-errors for Example 1.}\label{eg1}
\end{table}

\noindent {\bf Example 2. }The second group contains the examples with
\begin{eqnarray}
&&\zeta(t)=1.5+\sin t/3-\cos(2t)/4;\\
&&p(t)=\sin(2t)/20+\sin(\pi t+0.1)/20.
\end{eqnarray}
The surface $\Gamma$ is shown in Figure \ref{fig:periodic} and the graph of $p$ is shown in Figure \ref{fig:p2}.
In this case, the total fields no longer have analytical representations, thus we could only use numerical results with fine enough meshes to approximate the exact functions. For the numerical results, we fix $h=0.1$. For the "exact" solution, we set $N=256$, while for examples, we choose $N=8,16,32,64$. The relative errors are listed in Example \ref{eg2}.

\begin{figure}[htb]
\centering
\includegraphics[width=1\textwidth]{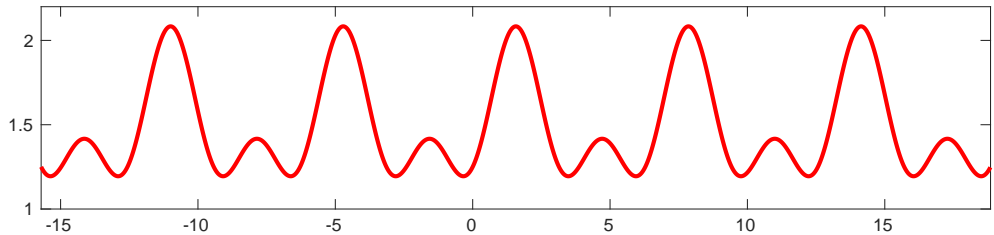} 
\caption{Periodic surface $\Gamma$.}
\label{fig:periodic}
\end{figure}

\begin{figure}[htb]
\centering
\includegraphics[width=1\textwidth]{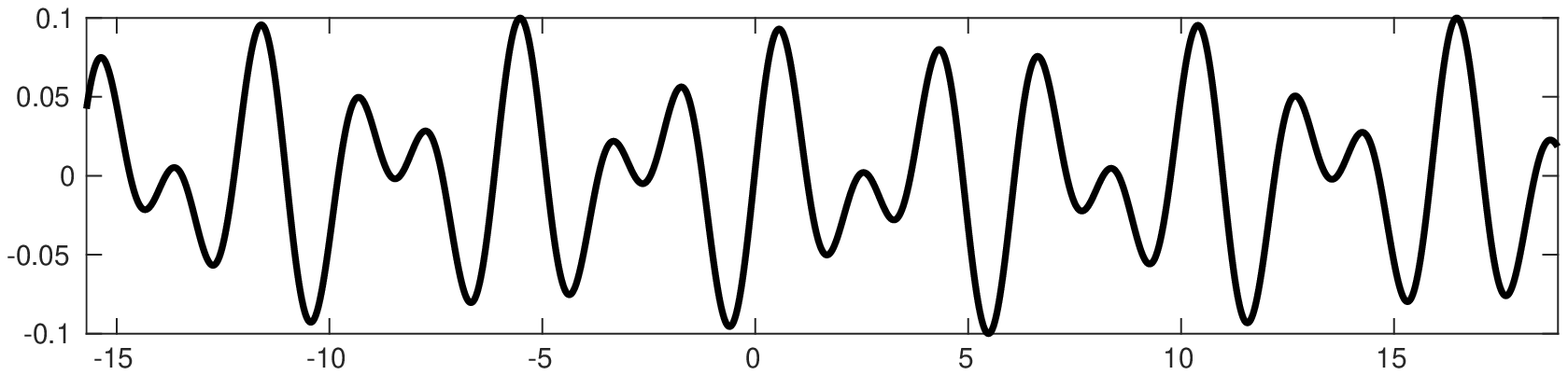} 
\caption{Graph of $p$ in Example 2.}
\label{fig:p2}
\end{figure}

\begin{table}[htb]
\centering
\begin{tabular}
{|p{1.8cm}<{\centering}||p{2cm}<{\centering}|p{2cm}<{\centering}
 |}
\hline
  & $k=1$ & $k=\sqrt{2}$\\
\hline
\hline
$N=8$&$3.71$E$-01$&$7.65$E$-01$\\
\hline
$N=16$&$1.16$E$-02$&$2.09$E$-01$\\
\hline
$N=32$&$4.17$E$-05$&$3.95$E$-04$\\
\hline
$N=64$&$3.50$E$-08$&$8.50$E$-07$\\
\hline
\end{tabular}
\caption{Relative $L^2$-errors for Example 2.}\label{eg2}
\end{table}

\noindent {\bf Example 3. }The second group contains the examples with
\begin{eqnarray}
&&\zeta(t)=1.5+\sin t/3-\cos(2t)/4;\\
&&p(t)=\sin\left[(4+t^2)^{1/3}\right]/20.
\end{eqnarray}
The graph of $p$ is shown in Figure \ref{fig:p3}.
We use the same parameters as Example 2, except for $k=1,1.5$. The relative errors are listed in Example \ref{eg3}.

\begin{figure}[htb]
\centering
\includegraphics[width=1\textwidth]{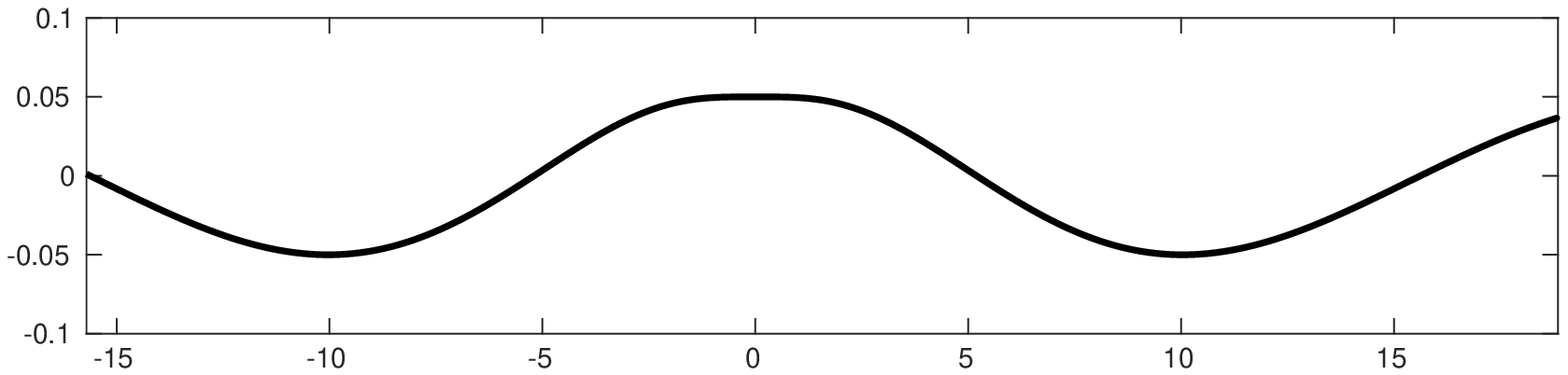} 
\caption{Graph of $p$ in Example 3.}
\label{fig:p3}
\end{figure}

\begin{table}[htb]
\centering
\begin{tabular}
{|p{1.8cm}<{\centering}||p{2cm}<{\centering}|p{2cm}<{\centering}
 |}
\hline
  & $k=1$ & $k=1.5$\\
\hline
\hline
$N=8$&$3.74$E$-01$&$3.86$E$-01$\\
\hline
$N=16$&$1.14$E$-02$&$9.53$E$-04$\\
\hline
$N=32$&$4.21$E$-05$&$1.07$E$-06$\\
\hline
$N=64$&$3.59$E$-08$&$1.23$E$-08$\\
\hline
\end{tabular}
\caption{Relative $L^2$-errors for Example 3.}\label{eg3}
\end{table}

The logarithmic scale of the relative $L^2$-errors for Example 2 and 3. The slopes for the examples are roughly about $-8$ to $-9$, which means that the numerical results converges at least at the rate of $N^{-8}$. However, as $n=5$, from Theorem \ref{th:converge_final} is about $O(N^{-5})$. The convergence rate of the numerical results is even greater than expected. 

\begin{figure}[htb]
\centering
\begin{tabular}{c c}
\includegraphics[width=0.4\textwidth]{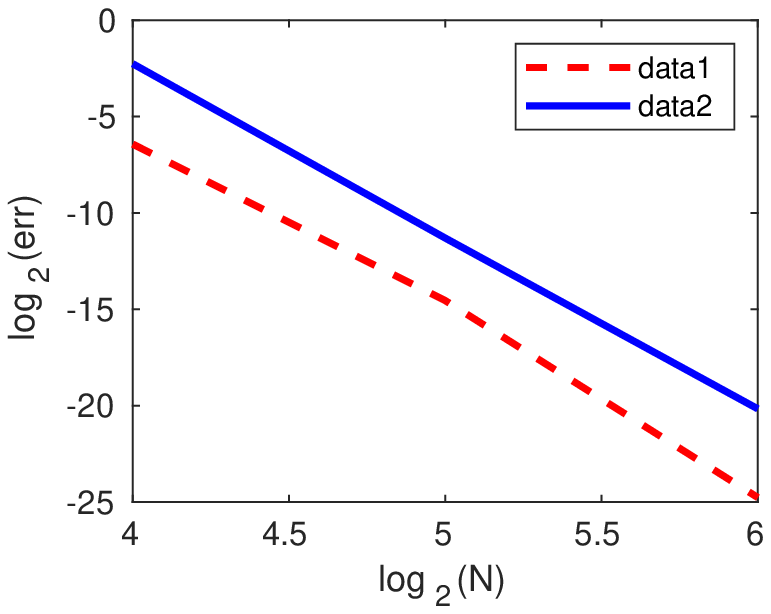} 
& \includegraphics[width=0.4\textwidth]{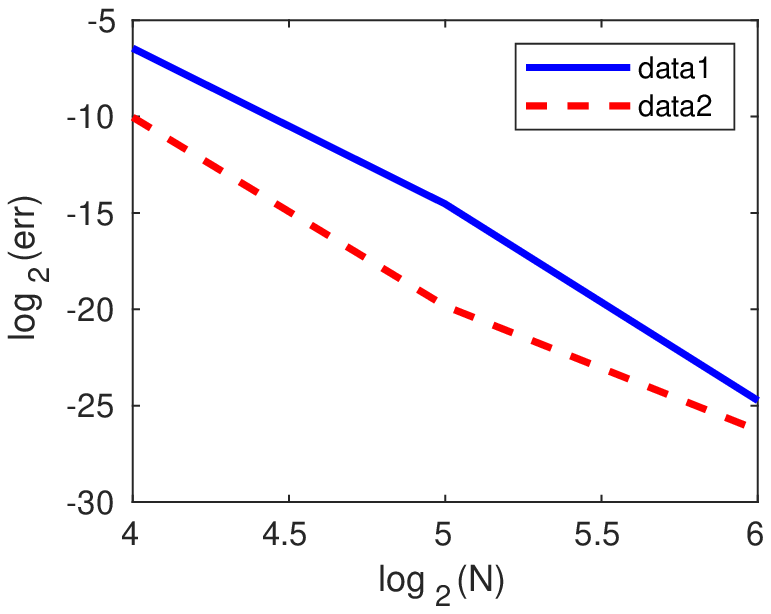}\\[-0cm]
(a) & (b)  
\end{tabular}%
\caption{(a): Example 2; (b): Example 3.}
\label{fig:surface}
\end{figure}

\section*{Appendix}

\subsection*{ The Floquet-Bloch transform}

The main tool used in this paper is the Floquet-Bloch transform. In this section, we will recall the definition and some basic properties of the Bloch transform in periodic domains in $\R^d$ (for details see \cite{Lechl2016}).

Suppose $\Omega\subset\R^d$ is $\Lambda$-periodic in $\widetilde{x}$-direction, i.e., for any $\widetilde{x}^\top=(\widetilde{x}^\top,x_d)\in\Omega$, the translated point $x+\left(\begin{smallmatrix}\Lambda j\\0 \end{smallmatrix}\right)\in\Omega,\,\forall{ j}\in\Z$. 
Define one periodic cell by $\Omega^\Lambda:=\Omega\cap\left[\W\times\R\right]$. For any $\phi\in C_0^\infty(\Omega)$, define the (partial)  Bloch transform in $\Omega$, i.e., $\J_{\Omega}$, of $\phi$ as
\begin{equation*}
\left(\J_\Omega\phi\right)({\alpha},{ x})=C_\Lambda\sum_{{ j}\in\Z^{d-1}}\phi\left({ x}+\left(\begin{matrix}
\Lambda { j}\\0
\end{matrix}\right)\right)e^{-\i{\alpha}\cdot\Lambda{ j}},\quad {\alpha}\in\R ,\,{ x}\in\Omega^\Lambda
\end{equation*}
where $C_\Lambda=\sqrt{\frac{|\Lambda|}{2\pi}}$. The Bloch transform is well-defined for any smooth function with compact support, and could be extended to more generalized Sobolev spaces.

\begin{remark}
The periodic domain $\Omega$ is not required to be bounded in $x_d$-direction.  
\end{remark}

We can also define the weighted Sobolev  space on the unbounded domain $\Omega$ by
\begin{equation*}
H_r^s(\Omega):=\left\{\phi\in \mathcal{D}'(\Omega):\,(1+|{ x}|^2)^{r/2}\phi({ x})\in H^s(\Omega)\right\}.
\end{equation*}
For any $\ell\in\N$, $s\in\R$, we can also define the following Hilbert space by
\begin{equation*}
H^\ell(\Wast;H^s(\Omega^\Lambda)):=\left\{\psi\in\mathcal{D}'(\Wast\times\Omega^\Lambda):\,\sum_{m=0}^\ell\int_\Wast\left\|\partial^m_{\alpha}\psi({\alpha},\cdot)\right\|\d{\alpha}<\infty\right\},
\end{equation*}
and extend the index $\ell\in\N$ to any $r\in\R$ by interpolation and duality arguments. The space $H_0^r(\Wast;H_\alpha^s(\Omega^\Lambda))$ is the subspace of $H^\ell(\Wast;H^s(\Omega^\Lambda))$ such that all functions in the space is periodic with respect to $\alpha$ and quasi-periodic with respect to the second variable. The following properties for the $d$-dimensional (partial) Bloch transform $\J_\Omega$ is also proved in \cite{Lechl2016}.

\begin{theorem}\label{th:Bloch_property}
The Bloch transform $\J_\Omega$ extends to an isomorphism between $H_r^s(\Omega)$ and $H_0^r(\Wast;H_\alpha^s(\Omega^\Lambda))$ for any $s,r\in\R$. Its inverse has the form of
\begin{equation*}
(\J^{-1}_\Omega\psi)\left({ x}+\left(\begin{matrix}
\Lambda { j}\\0
\end{matrix}\right)\right)=C_\Lambda\int_\Wast \psi({\alpha},{ x})e^{\i{\alpha}\cdot\Lambda{ j}}\d{\alpha},\quad x_1\in\Omega^\Lambda,\,{ j}\in\Z,
\end{equation*}
and the adjoint operator $\J^*_\Omega$ with respect to the scalar product in $L^2(\Wast;L^2(\Omega^\Lambda))$ equals to the inverse $\J^{-1}_\Omega$. Moreover, when $r=s=0$, the Bloch transform $\J_\Omega$ is an isometric isomorphism.
\end{theorem}

Another important property of the Bloch transform is the commutes with partial derivatives, see \cite{Lechl2016}. If $u\in H_r^n(\Omega)$ for some $n\in\N$, then for any ${ \gamma}=(\gamma_1,\dots,\gamma_d)\in\N^d$ with $|\gamma|=\leq N$,
\begin{equation*}
\partial^{ \gamma}_{  x} \left(\J_\Omega u\right)({ \alpha},{  x})=\J_\Omega[\partial^{ \gamma} u]({ \alpha},{  x}).
\end{equation*}

\begin{remark}
The definition of the partial Bloch transform could also be extended to other periodic domains, for example, periodic hyper-surfaces. If $\Gamma$ is a $\Lambda$-periodic surface defined in $\R^d$, then we can define $\J_\Gamma$ in the same way, and obtain similar properties. In this paper, we will denote the Bloch transform $\J_X$ by the partial Bloch transform in the domain $X\subset\R^d$, which is periodic with respect to $\widetilde{x}$-direction.
\end{remark}

\begin{remark}\label{rem:Four}
There is an alternative definition for the space $H_0^r(\Wast;X_{ \alpha})$, where $X_{ \alpha}$ is a family of Hilbert spaces that are ${ \alpha}$-quasi-periodic in $\widetilde{  x}$. Let 
\begin{equation*}
\phi_{\Lambda^*}^{({  j})}({ \alpha})=C_\Lambda e^{-\i{ \alpha}\cdot\Lambda {  j}},\,{  j}\in\Z
\end{equation*}
be a complete orthonormal system in $L^2(\Wast)$, then any function $\psi\in \mathcal{D}'(\Wast\times\Omega^\Lambda)$ has a Fourier series
\begin{equation*}
\psi({ \alpha},{  x})=C_\Lambda \sum_{{ \ell}\in\Z}\hat{\psi}_{\Lambda^*}({ \ell},{  x})e^{-\i{ \alpha}\cdot\Lambda{ \ell}},
\end{equation*}
where $\hat{\psi}_{\Lambda^*}({ \ell},{  x})=<\psi(\cdot,{  x}),\phi_{\Lambda^*}^{({ \ell})}>_{L^2(\Wast)}$. Then the squared norm of any $\psi\in H_0^r(\Wast;X_{ \alpha})$ equals to
\begin{equation*}
\|\psi\|^2_{H_0^r(\Wast;X_{ \alpha})}=\sum_{{ \ell}\in\Z}(1+|{ \ell}|^2)^r\left\|\hat{\psi}_{\Lambda^*}({ \ell},\cdot)\right\|^2_{X_{ \alpha}}.
\end{equation*}
\end{remark}

\subsection{Cutoff functions}

We look for a smooth cutoff functions $g$ in the interval $[a,b]$ such that
\begin{equation}
g(a)=0,\,g(b)=1,\, g^{(j)}(a)=g^{(j)}(b)=0\text{ for any }j=1,2,\dots,n,
\end{equation}
where $n\in\N^+$ is a positive integer.
We can define the function by
\begin{equation}
g(t)=c\int_a^t (\tau-a)^{n+1}(b-\tau)^{n+1}\d\tau,
\end{equation}
where the constant $c$ is defined by
\begin{equation}
c=\left(\int_a^b (\tau-a)^{n+1}(b-\tau)^{n+1}\d\tau\right)^{-1}.
\end{equation}

If the cutoff function satisfies
\begin{equation}
g(a)=0,\,g(b)=1,\, g^{(j)}(a)=g^{(j)}(b)=0\text{ for any }j=1,2,\dots,\infty,
\end{equation}
we can define 
\begin{equation}
g(t)=c\int_a^t \exp\left[-\frac{1}{((a+b)/2)^2-(\tau-(b-a)/2)^2}\right]\d\tau,
\end{equation}
where
\begin{equation}
c=\left(\int_a^b \exp\left[-\frac{1}{((a+b)/2)^2-(\tau-(b-a)/2)^2}\right]\d\tau\right)^{-1}.
\end{equation}

\bibliographystyle{alpha}
\bibliography{ip-biblio} 

\end{document}